\documentclass[11pt]{amsart}

\usepackage[T1]{fontenc}
\usepackage[margin=1.17in]{geometry}
\usepackage{amsmath,amssymb,mathtools,mathrsfs}
\usepackage{enumitem}
\usepackage{microtype}
\usepackage[colorlinks=true,
  citecolor=blue,
  linkcolor=blue,
  urlcolor=blue]{hyperref}

\allowdisplaybreaks[2]
\numberwithin{equation}{section}

\newtheorem{theorem}{Theorem}[section]
\newtheorem{proposition}[theorem]{Proposition}
\newtheorem{lemma}[theorem]{Lemma}
\newtheorem{corollary}[theorem]{Corollary}
\theoremstyle{definition}

\newtheorem{example}[theorem]{Example}
\theoremstyle{remark}
\newtheorem{remark}[theorem]{Remark}

\newcommand{\E}{\mathbb E}
\newcommand{\PP}{\mathbb P}
\newcommand{\N}{\mathbb N}
\newcommand{\R}{\mathbb R}
\newcommand{\Z}{\mathbb Z}
\newcommand{\ind}{\mathbf 1}
\newcommand{\Pois}{\operatorname{Poisson}}
\newcommand{\Var}{\operatorname{Var}}

\newcommand{\cum}{\textsf{cum}}
\newcommand{\Inf}{\operatorname{Inf}}

\newcommand{\dTV}{d_{\mathrm{TV}}}
\newcommand{\cL}{\mathcal L}
\newcommand{\cP}{\mathcal P}
\newcommand{\cR}{\mathcal R}
\newcommand{\cD}{\mathfrak D}
\newcommand{\cZ}{\mathcal Z}
\newcommand{\cC}{\mathcal C}
\newcommand{\cT}{\mathcal T}
\newcommand{\cG}{\mathcal G}
\newcommand{\cM}{\mathcal M}

\newcommand{\cJ}{\mathcal J}
\newcommand{\sZ}{\mathscr Z}
\newcommand{\fH}{\mathfrak H}
\newcommand{\fE}{\mathfrak E}
\newcommand{\noi}{\noindent}
\newcommand{\bul}{\bullet}
\newcommand{\dd}{\,\mathrm d}
\newcommand{\ot}{\otimes}
\newcommand{\sot}{\widetilde\otimes}
\newcommand{\norm}[1]{\left\lVert #1\right\rVert}
\newcommand{\abs}[1]{\left\lvert #1\right\rvert}

 \newcommand{\eps}{\varepsilon}

\newcommand{\AND}{\quad{\rm and}\quad}

\hypersetup{
  pdftitle={Four-moment criteria for Poisson convergence on Poisson
    and Rademacher chaoses},
  pdfauthor={Guangqu Zheng}
}

\makeatletter
\@namedef{subjclassname@2020}{\textup{2020} Mathematics Subject Classification}
\makeatother

\title[FMC  for Poisson convergence on Poisson and Rademacher chaoses]
{Four-moment criteria for Poisson convergence\\
on Poisson and Rademacher chaoses}

\author{Guangqu Zheng}

\address{Department of Mathematics and Statistics,
Boston University, 665 Commonwealth Avenue,   
Boston, MA 02215, USA}
\email{gzheng90@bu.edu}

\date{\today}



\begin{document}

\begin{abstract}
In this paper, 
we establish Poisson limit theorems on Poisson and Rademacher
chaoses. Our principal result is a total-variation bound,
 valid in both settings, for an integer-valued functional
whose highest-order chaos is dominant. The approximation error is the
sum of the pure-chaos four-moment terms and an explicit remainder
controlled by the $L^4$-size of the lower-order chaoses. On Poisson
space the pure-chaos term is controlled solely by the moment defect,
whereas on Rademacher space an additional maximal-influence correction
is required. When the lower-order remainder vanishes, we recover
exactly the corresponding total-variation bound for a shifted pure
chaos. As consequences, for nonnegative integer-valued shifts of
random variables in a fixed Poisson chaos, convergence of the first
four moments is equivalent to convergence in total variation to a
Poisson law, together with uniform integrability of the fourth powers;
on a fixed Rademacher chaos, the analogous conclusion holds under a
vanishing maximal-influence condition.

The proof reveals a {\it unit-jump rigidity} phenomenon: the
four-moment defect simultaneously suppresses unwanted spectral
components and rules out non-unit jumps. We show, through an explicit
quadratic counterexample, that the maximal-influence condition in the
Rademacher setting is necessary for a general Poisson limit theorem.
Finally, for every order $q\geq2$, we construct pure Poisson-chaos
sequences with vanishing moment defect that converge weakly to a
centered Poisson law. These examples show that the exact higher-order
rigidity is not uniform once the lattice condition is removed. In both
the Poisson and Rademacher settings, our proofs follow a unified
strategy combining the Chen--Stein method, exchangeable pairs, and
Ledoux's spectral viewpoint.

\end{abstract}

\maketitle


\section{Introduction and main results}

The law of small numbers asserts that the
$\text{Binomial}(n,p)$ distribution is close to the
$\text{Poisson}(\lambda)$ distribution, where
$\lambda\in(0,\infty)$, whenever
$(n, p, np)\to(+\infty, 0,  \lambda)$.
Up to its deterministic mean, a $\text{Binomial}(n,p)$ random variable belongs to the
Rademacher chaos of order $1$. 
Indeed, if $B_1,\ldots,B_n$ are independent Bernoulli random variables
with parameter $p$, then

\noi
\begin{align*}
\text{Binomial}(n,p)
=np+\sqrt{p(1-p)}
\sum_{k=1}^nY_k,
\qquad
Y_k:=\frac{B_k-p}{\sqrt{p(1-p)}}.
\end{align*}

\noi
This elementary fact naturally suggests studying Poisson
approximation beyond linear statistics, and in particular Poisson limit
theorems on shifted Rademacher chaoses of order $d\geq 1$.
A parallel problem arises on the Poisson space. Inspired by the
fourth-moment phenomena in the Gaussian, Poisson, and Rademacher settings
\cite{NP05, DP18a, DVZ18, DK19, Zhe19}, we address the following more
ambitious question.

\medskip

 {\bf Question ($\bigstar$):} Can one establish four-moment criteria
({\bf FMC}) for Poisson convergence on shifted Poisson and
Rademacher chaoses?

\medskip

In this work, we answer  Question ($\bigstar$) affirmatively.
By contrast, owing to the diffusive  nature of Gaussian Wiener chaos,
we do not expect an analogous Poisson convergence result in the Gaussian
setting; see Remark \ref{rem_no_G}.

\medskip

In Section \ref{SEC_11}, we briefly review fourth-moment phenomena for normal approximation. 
We then state the common dominant-chaos theorem in Section \ref{SEC_12},
followed by the pure Poisson- and Rademacher-chaos criteria in
Sections \ref{SEC_12} and \ref{SEC_13}, respectively.

\subsection{Fourth moment phenomena for normal approximation}
\label{SEC_11}

The term ``chaos'' refers to an orthogonal decomposition of an
$L^2$-space generated by a given source of randomness, notably Gaussian,
Poisson, or Rademacher processes. Formally, one has
\[
L^2(\mathbb{P})
=
\bigoplus_{q\geq 0}\cC_q,
\]
where $\cC_q$ is called the $q$-th chaos. More precisely, 

\smallskip
\noi
(i) An isonormal Gaussian process gives rise to the
Wiener--It\^o decomposition into multiple Gaussian integrals.
The chaos $\cC_q$ may also be represented in terms of Hermite
polynomials of independent standard Gaussian random variables;
see, e.g., \cite[Definition 2.2.3]{NP12}.

\smallskip
\noi
(ii)  A Poisson random measure gives rise to the Poisson
Wiener--It\^o decomposition into multiple Poisson integrals, which is
closely related to Charlier polynomials; see, e.g.,
\cite[Proposition 6.2.9]{Pri09} and \cite{LP18}.

\smallskip
\noi
(iii) An independent sequence of two-point random variables gives rise
to the Walsh--Rademacher decomposition into homogeneous sums;
see, e.g., \cite[Section 6]{Pri08}.

In each setting, $\cC_q$ is the eigenspace of the corresponding
Ornstein--Uhlenbeck generator associated with the eigenvalue $-q$.
This common spectral structure is one of the fundamental principles of modern stochastic analysis.

\medskip
\noi
$\bul$ {\bf Normal approximation on a Gaussian space.}
 In their seminal paper \cite{NP05},   Nualart and
  Peccati proved that,
for a variance-one sequence belonging to a fixed Gaussian Wiener chaos,
convergence of the fourth moment to three is equivalent to convergence in
distribution to a standard normal. 
Peccati and Tudor
\cite{PT05} established the corresponding multivariate result.
An important methodological advance was the characterization in terms of
Malliavin derivatives obtained by Nualart and Ortiz-Latorre
\cite{NOL08}. Nourdin and Peccati \cite{NP09} subsequently combined
Malliavin calculus with Stein's method of normal approximation to 
quantify the fourth moment theorem, among many other things.
The resulting Malliavin--Stein theory is developed systematically in the
monograph \cite{NP12}. Later works established optimal rates
\cite{NP15}, entropy and information-theoretic refinements \cite{NPS14, LNP15}, 
and interpretations based on Markov generators.
In particular, Ledoux \cite{Led12} and Azmoodeh, Campese, and Poly
\cite{ACP14} showed that, in a diffusive setting, spectral information
and the carr\'e du champ can replace explicit contraction computations.
Bourguin, Campese, Leonenko, and Taqqu \cite{BCLT19} subsequently
established four-moment theorems for Pearson targets on Markov chaoses.
More recently, Basse-O'Connor, Kramer-Bang, and Svendsen \cite{BKS25}
began extending fourth-moment criteria beyond a single eigenspace by
treating certain finite and infinite sums of Gaussian multiple
integrals; see also a refinement in \cite{DDN26}.
The historical account in \cite[Part~I]{ZheThesis} organizes several
developments prior to 2018 around the interplay among exchangeable pairs,
the Ornstein--Uhlenbeck generator, and the carr\'e du champ. Two points
from that thesis are especially relevant here. First, the
Nualart--Ortiz-Latorre criterion explains how a moment identity can
control a Malliavin quantity. Second, the exchangeable-pair construction
extends to non-diffusive spaces, although the failure of the derivation
property produces additional jump terms. The continuous
exchangeable-pair realization on Wiener chaos was developed by Nourdin
and the author \cite{NZ20}; its Poisson and Rademacher counterparts
provide the couplings used below.
 The present paper continues this program for a Poisson target.
The Stein operator for normal approximation is replaced by the Chen--Stein
operator for Poisson approximation, while the non-diffusive chain-rule
remainder is identified with a {\it spectral defect} measuring the
failure of the relevant increments to be unit jumps; 
see Proposition  \ref{prop_J}.
This provides the common mechanism 
underlying the two probability spaces considered here.

\medskip
\noi
$\bul$ {\bf Normal approximation on a Poisson space.}
On the Poisson space, early Malliavin--Stein estimates were obtained by
Peccati, Sol\'e, Taqqu, and Utzet \cite{PSTU10}, and by Peccati and
Zheng \cite{PZ10}. Early attempts to establish fourth-moment theorems
(for a Gaussian target)
and related quantitative bounds on Poisson chaoses include
\cite{PT08,LRP13,ET14,PZ14,Sch16}; 
these works treated double,  triple integrals, 
or kernels satisfying additional structural assumptions.
D\"obler and Peccati \cite{DP18a} proved the first general
fourth-moment theorem on a fixed Poisson chaos by combining Ledoux's
spectral approach, Mecke-type identities, and discrete Malliavin--Stein method.
Their qualitative and quantitative results were obtained under certain
technical assumptions. D\"obler, Vidotto, and the author \cite{DVZ18}
subsequently gave an exchangeable-pair proof under the sole assumption
of a finite fourth moment. Their work also contains a multivariate
theorem,   a Peccati-Tudor type result. 
D\"obler and
Peccati \cite{DP18b} then completed the analytic picture by deriving
necessary and sufficient contraction and carr\'e-du-champ criteria from
a product formula valid under minimal integrability assumptions. 
See also \cite{CP25} for further product formulas in this direction.
More recently, the author \cite{Zhe26} established a Kolmogorov bound
for the fourth-moment theorem under the sole assumption of a finite
fourth moment, thereby complementing the quantitative estimates in
\cite{DP18a,DVZ18}. Related results for non-Gaussian targets include the
four-moment Gamma theorems in \cite{FT16,DP18a}.

\medskip
\noi
$\bul$ {\bf Normal approximation in a Rademacher setting.}
The discrete Malliavin--Stein method was
initiated in \cite{NPR10} for sequences of independent symmetric
Rademacher random variables, and was subsequently extended to the
non-symmetric and non-homogeneous setting in \cite{KRT16,KRT17, Zhe17}.
D\"obler and Krokowski \cite{DK19} proved a fourth-moment theorem on
Rademacher chaos under an additional maximal-influence condition on the
kernels. They also constructed counterexamples showing that 
an unrestricted fourth-moment theorem fails on every symmetric Rademacher chaos of order at least two. 
The maximal-influence condition is reminiscent of the Lindeberg-type assumption 
in de Jong's central limit theorem for generalized multilinear forms \cite{deJong90}. 
The author \cite{Zhe19} subsequently used a coordinate-refresh exchangeable pair
to establish a Peccati--Tudor-type theorem in the Rademacher setting.

\medskip
\noi
$\bul$ {\bf Central theme of this paper: Poisson approximation.}
Peccati \cite{Pec12} combined discrete Malliavin
calculus on the Poisson space with the Chen--Stein method \cite{Chen75}.
As an application of his general Malliavin bounds  for Poisson approximation, 
he  treated an integer-valued 
perturbed chaos consisting of a fixed-order multiple integral, a deterministic
shift and lower-order terms.  
Bourguin and Peccati \cite{BP14} proved
portmanteau inequalities for mixed Gaussian--Poisson targets and
multidimensional clustering.  Privault and Torrisi \cite{PT15} and
Krokowski \cite{Kro17} developed corresponding bounds on Bernoulli and
Rademacher product spaces.  Decreusefond, Schulte, and Th\"ale
\cite{DST16} established functional Poisson approximation for Poisson
and binomial $U$-statistics, with applications throughout stochastic
geometry.  Recent work of Liu and Privault \cite{LP25} identifies a
normal-to-Poisson phase transition for subgraph counts in a  
random-connection model. 
These results are typically expressed
through Malliavin gradients, contractions, cumulants, or local dependence. 
Let us also mention the work  \cite{Her20}
by  Herry, who developed an abstract stable-limit theory for
Poisson functionals with conditionally Gaussian or Poisson targets. In
the Poisson case, his criterion combines convergence of a Malliavin
covariance term to a possibly random intensity with the asymptotic
unit-jump condition
\[
\E\int | u_n(z)D^+_zF_n(D^+_zF_n -1) | \nu(dz)\to 0
\]
where $F_n = \delta(u_n)$ is the Skorokhod divergence of $u_n$;
see \cite[($\mathrm{P}_3$)-($\mathrm{P}_4$)]{Her20}.
The   term $D^+_zF_n(D^+_zF_n-1)$, with $D^+_z$ discrete Malliavin derivative,
 is closely related to the spectral defect  $\cJ(F)$  appearing
in \eqref{J_limit}: both quantify the failure of Malliavin increments to behave as
unit jumps. Herry's result applies to general Poisson functionals and
yields stable convergence to Poisson mixtures, but it is qualitative
and is not formulated as a moment criterion. Our results specialize to
shifted random variables in finitely many chaoses 
and show that the relevant Malliavin conditions
can be captured by four moments, while also yielding quantitative
bounds in total variation. Finally, let us point out 
 a related noncommutative precedent,  the mixed
third--fourth-moment criterion of Nourdin and Peccati \cite{NPfree13}
for centered free Poisson limits on a fixed Wigner chaos; Bourguin
\cite{Bou15} obtained an analogous fourth-moment criterion on the free
Poisson algebra.

The Question $(\bigstar)$ here is classical and different:
When does the four-moment relation suffice to force Poisson
convergence?
Our answer to Question $(\bigstar)$ has two main parts:

\smallskip
(i) On both Poisson and Rademacher spaces, an integer-valued functional
with a dominant highest-order chaos satisfies a  
four-moment bound in total variation, together with an explicit
remainder measuring the lower-order chaoses; see
Theorem~\ref{thm_pert}.  On Poisson space the leading-chaos error is
moment-only, whereas on Rademacher space it contains 
an additional  maximal-influence correction.

\smallskip
(ii) Setting the lower-order remainder equal to zero gives exactly the
pure-chaos total-variation bounds and the {\bf FMC}  in
Theorems~\ref{thm_P} and \ref{thm_R}.  In the Rademacher case, an exact
quadratic counterexample (Proposition~\ref{prop_counter}) shows that
the additional maximal-influence condition cannot in general be
removed.

\medskip

The lattice condition is the mechanism that turns the moment relations
into the one-unit-jump structure required by the Chen--Stein method.
In Theorem~\ref{thm_pert}, it is imposed on the full functional,
whereas the leading chaos itself need not be lattice-valued.  The pure
chaos bounds are recovered by setting the lower-order remainder equal
to zero.  Section~\ref{SEC_43} complements these results with two pure
higher-order Poisson examples.  Once the exact lattice condition is
removed, one may still have vanishing moment defect
and Poisson law in the limit.
This distinction between an exact lattice criterion and limiting
dominant-chaos behavior guides the organization below.

\medskip

Before stating the main results, we fix some notations and 
introduce the lattice condition. 

\medskip
\noi
{\bf $\bul$ Notations.} Let $\N:=\{1,2,\ldots\}$, $\N_0:=\{0\}\cup\N$,  $m_k(X):=\E[X^k]$, 
and
let $\cL(X)$ denote   the law of a real-valued random variable $X$.
For $r\in\N$,  if $X$ has a finite $r$-th moment, its $r$-th
cumulant is denoted by $\cum_r(X)$ and  in particular, if $X$ is centered, 
we have 

\noi
\begin{align}
 \cum_2(X)=m_2(X),
 \quad
 \cum_3(X)=m_3(X)
 \AND \cum_4(X)=m_4(X)-3m_2(X)^2.
\notag 
\end{align}
Cumulants are additive  under independent sums.  
Moreover, if
$P_\lambda\sim\Pois(\lambda)$, then
$\cum_r(P_\lambda)=\lambda$ for every $r\geq 1$; this notation $P_\lambda$
 for a
Poisson random variable will be used throughout the paper.
The first four   moments of $P_\lambda$ are

\noi
\begin{align}
 \lambda,
 \quad \lambda^2+\lambda,
 \quad \lambda^3+3\lambda^2+\lambda,
 \quad \lambda^4+6\lambda^3+7\lambda^2+\lambda,
 \label{4M}
\end{align}
and its centered moments of orders $2,3,4$ are
\begin{align}\label{center_m}
\lambda,  \quad  \lambda,  \quad 3\lambda^2+\lambda.
\end{align}
For probability measures $\nu_1,\nu_2$ on $\N_0$, 
we define the total-variation distance
by 

\noi
\begin{align}
 \dTV(\nu_1,\nu_2)
 :=\sup_{A\subseteq\N_0}
 \big|\nu_1(A) - \nu_2(A) \big|.
 \label{def_TV}
\end{align}
Often we write $\dTV(X, Y)$ for $\dTV(\cL(X), \cL(Y))$.
We also write $\|X\|_s$ for the $L^s(\PP)$-norm of $X$.

\medskip

 A basic and natural assumption throughout this work
 is the following {\it lattice condition}.
 
 \smallskip
 
\noi
{\bf $\bul$ Lattice condition.}  A centered random variable $F$ is said to satisfy the lattice condition 
with shift $\theta$
if   $\theta \in \R$ is deterministic 
and

\noi
\begin{align}\label{cond_La}
\text{$F+ \theta\in\N_0$ almost surely.}
\end{align}
When $F+\theta$ is compared with $\Pois(\theta)$,
we additionally assume $\theta\in(0, \infty)$.
See Remark \ref{rem_lattice} for its relation to the 
unit-jump rigidity.

\subsection{Dominant-chaos approximation in the two settings}
\label{SEC_12}

Let $(\cZ,\sZ,\mu)$ be a $\sigma$-finite measure
space, let $\eta$ be a Poisson random measure 
with control measure $\mu$,
and let $\cC_{q}^{\mathsf{P}}$ denote its $q$-th Poisson chaos.
For a centered random variable $F\in L^4(\PP)$, define
the {\it moment defect}

\noi
\begin{align}
 \cP(F)
 :=m_4(F)-3m_2(F)^2-2m_3(F)+m_2(F),
 \label{def_P}
\end{align}
which vanishes when $F$ is a centered Poisson random variable
in view of \eqref{center_m}. 
Equivalently,  $\cP(F)  =\cum_4(F)-2\cum_3(F)+\cum_2(F).$
Since the cumulants are additive under independent sums, 
the moment defect $\cP$ is additive under independent sums.

We next fix the Rademacher setting. Let $(X_k)_{k\geq1}$ be an independent,
possibly non-symmetric and non-homogeneous Rademacher sequence 
such that

\noi
\begin{align}
\PP(X_{k}=1)=p_{k}\in(0,1)
\AND
\PP(X_{k}= -1)=q_{k}:=1-p_{k}.
\notag 
\end{align}

\noi
Let $Y=(Y_k)_{k\geq1}$ be its normalized version, given by

\noi
\begin{align}\label{Rad_norm}
Y_k:= \frac{ X_{k} -  p_{k} + q_{k} }{  2\sqrt{p_k q_k}  }.
\end{align}

\noi
Let $\fH:=\ell^2(\N)$ and for $d\geq1$, denote by
$\fH_0^{\odot d}$ the space of symmetric,
square-summable kernels on $\N^d$ that vanish on diagonals.
For a fixed $d\geq1$, let
$f\in\fH_0^{\odot d}$ and set, following \cite{Zhe19},

\noi
\begin{align}
F:=Q_d(f;Y)
:=\sum_{i_1,\ldots,i_d\in\N}  f(i_1,\ldots,i_d)Y_{i_1}\cdots Y_{i_d}.
\label{Rad_chaos}
\end{align}

\noi
Thus, $F$ belongs to the $d$-th Rademacher chaos $\cC^{\mathsf{R}}_{d}$,
where $\cC^{\mathsf{R}}_{d} =\{ Q_d(g; Y): g\in\fH_0^{\odot d} \}$.

For $k\in\N$, define the influence of the $k$-th coordinate for 
$F=Q_d(f;Y)$
 by
 
\noi
\begin{align}
\begin{aligned}
\Inf_k(F) \equiv \Inf_k(f)
&:=
 \sum_{i_1,\ldots,i_{d-1}\in\N}
f(i_1,\ldots,i_{d-1},k)^2,
\\
\cM(F)\equiv \cM(f)
&:=
\sup_{k\geq1}\Inf_k(f).
\end{aligned}
\label{Rad_inf}
\end{align}
For $d=1$, the first quantity is understood as
$\Inf_k(f)=f(k)^2$.
See also \eqref{Dk_rel}.

Finally, for every centered random variable $F$ with a finite fourth
moment, define

\noi
\begin{align}
\cR(F) := m_4(F)-3m_2(F)^2 - 2m_3(F)+m_2(F).
\label{def_R}
\end{align}
The quantity $\cR(F)$ has the same algebraic form as $\cP(F)$ in
\eqref{def_P}. Unlike its Poisson counterpart, however,
$\cR(F)$ can be negative in the Rademacher setting.\footnote{Let 
$F$ be a symmetric
Rademacher variable, so that
$
 \PP(F=1)=\PP(F=-1)=\frac12.
$
Then $m_2(F)=m_4(F)=1$ and $m_3(F)=0$, whence
$\cR(F)=1-3+1=-1.$ }

To state the two settings together, let
$\mathsf{S}\in\{{ \mathsf{P}},{\mathsf{R}}\}$.  We write
$\cC_{r}^{\mathsf{S}}$ for the $r$-th Poisson or Rademacher chaos,
respectively.  For a centered variable $F$ in the $r$-th chaos, put
\noi
\begin{align}
 \mathfrak d_{\mathsf{S}}(F)
 &:={}
 \begin{cases}
  \cP(F),&\mathsf{S}={ \mathsf{P}},\\[1mm]
  \abs{\cR(F)}
  +(2r)!\binom{2r}{r}m_2(F)\cM(F),&\mathsf{S}={\mathsf{R}},
 \end{cases}
\AND
 b_{\mathsf{S}}
 &:={}
 \begin{cases}
  \frac43,&\mathsf{S}={ \mathsf{P}},\\[1mm]
  2,&\mathsf{S}={\mathsf{R}}.
 \end{cases}
\end{align}
In the Poisson case, $\cP(F)\geq0$ by Lemma~\ref{lem_P_top} and \eqref{D_ide}.

The following quantitative result is the main theorem of the
paper.  It treats the two spaces simultaneously and allows the full
integer-valued functional to contain lower-order chaoses, while its
highest-order component carries the four-moment information.

\begin{theorem}[\textsf{Dominant chaos with lower-order perturbations}]
\label{thm_pert}

Fix $\mathsf{S}\in\{{ \mathsf{P}},{\mathsf{R}}\}$, $r\in\N$, and
$\lambda\in(0,\infty)$.  For each $n\in\N$, let\footnote{The sum $\sum_{j=1}^0$ is understood to be empty.}

\noi
\begin{align*}
 F_n\in\cC_{r}^{\mathsf{S}}\cap L^4(\PP),
 \quad
 B_n=\sum_{j=1}^{r-1}B_{n,j},
\AND
 B_{n,j}\in\cC_{j}^{\mathsf{S}}\cap L^4(\PP).
\end{align*} 

\noi
Suppose that $F_n+B_n$ satisfies the lattice condition
\eqref{cond_La} with shift $\theta_n>0$, and define
$X_n:=\theta_n+F_n+B_n$.  Put
$\eps_n:=\sum_{j=1}^{r-1}\norm{B_{n,j}}_4$.
For every $n$ such that $\theta_n>0$, with
$c_{\theta_n}:=(1-e^{-\theta_n})/\theta_n$, define

\noi
\begin{align}
 \mathfrak{R}_{\theta_n}(F_n,\eps_n)
 :=
 c_{\theta_n}\bigg[ 2 \|F_n\|_4\eps_n+\eps_n^2+\frac{\eps_n}{2}
 + 14\bigl(1+ \|F_n\|_4+\eps_n\bigr)^3\eps_n  \bigg]
 +\left(1+\frac{c_{\theta_n}}2\right)\eps_n.
 \label{pert_rem}
\end{align}
Then,

\noi
\begin{align} \label{pert_main_bdd}
\begin{aligned}
 \dTV\bigl( X_n,  P_{\theta_n}\bigr)  
   & \leq
 c_{\theta_n}\bigg[  | m_2(F_n)-\theta_n|  + \frac{2r-1}{2r}\sqrt{\mathfrak{d}_{\mathsf{S}}(F_n)} \\
 &\qquad\qquad\qquad  + \frac23\sqrt{m_2(F_n)\mathfrak d_{\mathsf{S}}(F_n)}
 +b_{\mathsf{S}}\mathfrak d_{\mathsf{S}}(F_n)  \bigg]  
   +\mathfrak R_{\theta_n}(F_n,\eps_n).
\end{aligned}
\end{align}

\noi
In particular, $\mathfrak R_{\theta_n}(F_n,0)=0$.  Hence, when
$\eps_n=0$, one has $B_n=0$ and \eqref{pert_main_bdd} reduces exactly
to Corollary~\ref{cor_P_quant} if $\mathsf{S}={ \mathsf{P}}$ and to
Corollary~\ref{cor_R_quant} if $\mathsf{S}={\mathsf{R}}$.

If, in addition,
$
 \big(\theta_n,  m_2(F_n), \mathfrak{d}_{\mathsf{S}}(F_n)\big)
 \xrightarrow{n\to+\infty}(\lambda, \lambda, 0)$
and
 $\eps_n\to0$,
then
$\dTV(X_n, P_\lambda)\to0$,
the first four moments of $X_n$ converge to those of
$P_\lambda$, and $(X_n^4)_{n\geq1}$ is uniformly integrable.
In the Rademacher case, under $m_2(F_n)\to\lambda$, the condition
$\mathfrak{d}_{\mathsf{R}}(F_n)\to0$ is equivalent to
$\cR(F_n)\to0$ and $\cM(F_n)\to0$.

\end{theorem}

As a consequence, we have the following     answer to Question $(\bigstar)$
in the Poisson setting.

\begin{theorem}[\textsf{Four-moment criterion on Poisson-chaos}]
\label{thm_P}
Fix $q\in\N$ and $\lambda \in(0, \infty)$.  Let
$F_n\in\cC_{q}^{ \mathsf{P}}\cap L^4(\PP)$ satisfy the lattice condition \eqref{cond_La}
with shift  $\theta_n$.
Then, $\cP(F_n )\geq0$ for every $n$, and the following assertions are
equivalent.

\smallskip
\noi
{\rm (i)}
The first four   moments of $F_n+ \theta_n$ converge to the four numbers in
\eqref{4M}.

\smallskip
\noi
{\rm (ii)}
One has
$\big(\theta_n, m_2(F_n),   \cP(F_n) \big) \to (\lambda,  \lambda,  0)$

\smallskip
\noi
{\rm (iii)}
One has
$\dTV\bigl(\cL(F_n + \theta_n), P_\lambda\bigr) \to0$
and the family $\{ (F_n + \theta_n)^4\}_{n\geq1}$ is uniformly integrable.
 
 In particular, convergence of the first four moments implies Poisson
convergence.
\end{theorem}

\begin{remark}  
(i) Because $F_n$ is centered,
$\theta_n=\E[F_n + \theta_n]$.  If $m_2(F_n)\to\lambda$, the condition
$\cP(F_n)\to0$ is equivalent to
 $m_4(F_n)-2m_3(F_n)
 \to3\lambda^2-\lambda.$
Thus the criterion consists of the mean, the variance, and one mixed
third-fourth moment statistic.  The nonnegativity of $\cP$ is not a
general moment inequality, but due to the    fact that
 the highest  chaotic  component $J_{2q}(F^2)$
  of $F^2$  has a second moment that exceeds $2m_2(F)^2$
(see Lemma \ref{lem_P_top} and \eqref{D_ide}-\eqref{E_abs}).

\smallskip
(ii)  Note that the exact lattice class is nonempty in every order.
If $A\in\sZ$ has integer measure and $N=\eta(A)$, then
 the monic $q$-th Charlier polynomial
$C_q(N, \mu(A))$ is an element of the $q$-th Poisson chaos; 
see \cite[Definition~6.2.7 and Proposition~6.2.9]{Pri09}.
Since $\mu(A)\in\N$, this polynomial is integer-valued on $\N_0$.
Being monic, it is   bounded below on $\N_0$, and hence an
integer shift makes it nonnegative.

\end{remark}

As shown in Example~\ref{ex_first_P} and
Remark~\ref{rem_exact_high_P}, a nondegenerate exact Poisson law can
arise as a shifted pure Poisson chaos only at order one.  Nevertheless,
the strict positivity in Remark~\ref{rem_exact_high_P} is not uniform
on normalized higher-order chaoses.  Examples~\ref{ex_rare_charlier}
and \ref{ex_geo_clique} construct, for every fixed $q\geq2$, 
random variables
$F_n\in\cC_q^{ \mathsf{P}}$ such that
$
 m_2(F_n)\to1$,
$
 \cP(F_n)\to0,$
and
$ 1+F_n\to \Pois(1)$ in law. 
See
Remark~\ref{rem_high_P_scope} for further discussion.

\subsection{Pure-chaos criterion and structural consequences on Rademacher space}
\label{SEC_13}

The Rademacher specialization of Theorem~\ref{thm_pert} gives the
quantitative pure-chaos bound.  The next theorem records the full
equivalence, including the converse statement with uniform
integrability.

\begin{theorem}[\textsf{Four-moment criterion on Rademacher chaos}]
\label{thm_R}
Fix $d\in\N$ and $\lambda\in(0,\infty)$.  
Let $F_n\in\cC_d^{\mathsf{R}} \cap L^4(\PP)$ satisfy
 the lattice condition \eqref{cond_La} with shift $\theta_n$.
Suppose that

\noi
\begin{align}
 \cM(F_n)\to0.
 \label{M_zero}
\end{align}
Then the following assertions are equivalent.

\smallskip
\noi
{\rm (i)}
The first four   moments of $F_n+ \theta_n$ converge to the four numbers in
\eqref{4M}.

\smallskip
\noi
{\rm (ii)}
One has
$\big(\theta_n, m_2(F_n),   \cR(F_n) \big) \to (\lambda,  \lambda,  0)$

\smallskip
\noi
{\rm (iii)}
One has
$\dTV\bigl(\cL(F_n + \theta_n), P_\lambda\bigr) \to0$
and the family $ \{ (F_n + \theta_n)^4\}_{n\geq1}$ is uniformly integrable.

\end{theorem}

\begin{remark}  
\label{rem_no_G}
(i)  [\textsf{No Gaussian-chaos analogue}]
Suppose that
$F_n$ belongs to a fixed finite sum of Gaussian Wiener chaoses and
$F_n$ converges in law to  $F_\infty$,
with $\Var(F_\infty)>0$.
By \cite{NP13TV}, the law of $F_\infty$ is absolutely continuous 
with respect to the Lebesgue measure
and
the convergence even holds in total variation.  Consequently,
$F_\infty$ cannot have a centered Poisson law.   
The discrete structures of Poisson and Rademacher spaces are therefore
indispensable for the Poisson limits studied here.

(ii) Unlike the Poisson setting, no nondegenerate exact Poisson law
can arise as a shift of a first-order Rademacher chaos; see
Proposition~\ref{prop_no_exact_R}.
\end{remark}

\begin{remark}[\textsf{The influence condition is structural}]

The maximal influence controls collisions (of indices in the kernels)
in the top spectral projection of $F_n^2$ (see Lemma \ref{lem_R_top}).
It is not a technical obstacle of the proof.
Proposition \ref{prop_counter} constructs a quadratic symmetric Rademacher 
chaos whose nonnegative shift has exactly the first four moments of a
Poisson law but is supported on the even integers.  This is the
Poisson-target counterpart of   \cite[Theorem 1.6]{DK19}.

\end{remark}

The maximal-influence condition also has a structural consequence
under the lattice assumption.  If the Bernoulli parameters remain
uniformly separated from zero and one, then an integer-valued shift of
a fixed-order Rademacher chaos can depend on only finitely many coordinates.
Consequently,
such a sequence cannot converge to a nondegenerate Poisson law.

\begin{theorem}[\textsf{Local bias bound and finite coordinate support}]
\label{thm_bias}
Let $F\in\cC_d^{\mathsf{R}}$ satisfy the lattice condition \eqref{cond_La}
with shift $\theta\in\R$.
 Assume that, for some
$\eps\in(0,\frac12]$,

\noi
\begin{align}
 \eps\leq p_k\leq1-\eps
 \qquad
 \text{for every $k$ with $\Inf_k(F) > 0$}.
 \label{bias_compact}
\end{align}
Then every  coordinate  $k$ with positive influence
satisfies

\noi
\begin{align}
 d\,d!\Inf_k(F)
 \geq\eps^d(1-\eps),
 \label{inf_lower}
\end{align}
and $F$ depends on at most
$
 \frac{d\,m_2(F)}
 {\eps^d(1-\eps)}
 $
coordinates.  In the symmetric case,
$
 \Inf_k(F)\geq \frac{2^{-(d+1)}}{d\,d!}
$
for every   coordinate $k$ with positive influence.  

Consequently, no sequence of integer-valued shifts
$F_n+\theta_n$ satisfying \eqref{bias_compact}, with fixed $d$,
a common $\eps>0$, and uniformly bounded variances, can converge to a
nondegenerate Poisson law.

\end{theorem}

The proof of Theorem \ref{thm_bias} is given in Section \ref{SEC_44}.

\medskip

In the following, we sketch our proof strategy 
and elaborate how the moment defect $\cP(F)$ forces
the Poisson convergence on the Poisson chaos,
while the argument for the Rademacher setting 
is analogous with the additional care on the maximal influence 
$\cM(F)$.

\subsection{A common spectral principle and lower-order perturbations}

The proofs of Theorem~\ref{thm_pert}, Theorem~\ref{thm_P},
and Theorem~\ref{thm_R} start from the exchangeable-pair bound in
the Chen--Stein method (Lemma~\ref{lem_Stein}).
Let $(X,X')$ be an exchangeable pair with values in $\N_0$,
meaning that $(X,X') = (X',X)$ in law.  
 Put
$D:=X'-X$, and suppose that
$\E[D |  X]=-a(X-\theta)+R$ for some $a>0$ and $\theta>0$.
With $c_\theta=(1-e^{-\theta})/\theta$, Lemma~\ref{lem_Stein} gives

\noi
\begin{align}\label{CS_bdd}
 \dTV\bigl( X,P_\theta\bigr)
 &\leq c_\theta\bigg\{
\underbrace{ \E\Big|  \frac{\E[D^2 |  X]}{2a}  -\frac{X+\theta}{2}  \Big|}_{(\clubsuit)}
 +   \underbrace{ \frac{\E[D^2(D^2-1)]}{6a}  }_{(\spadesuit)}  \bigg\} 
 +\left(1+\frac{c_\theta}{2}\right)
 \frac{\E|R|}{a}.
\end{align}
Thus the argument reduces to controlling three quantities: the
conditional quadratic-variation error $(\clubsuit)$,
the normalized non-unit-jump error $(\spadesuit)$,
and the regression error $\E|R|/a$.  For a pure-chaos
Ornstein--Uhlenbeck pair, the regression is exact and $R=0$.  In the
dominant-chaos setting on either space, the lower-order components
produce a small regression remainder, which is controlled by $\eps$
in \eqref{pert_R_bdd}.
To illustrate the main ideas, let us assume $\eps = 0$
for now.

\begin{remark}[\textsf{The lattice condition and  one-unit rigidity}]
\label{rem_lattice}
The lattice condition \eqref{cond_La} for $F$ means that
$
X = F+ \theta \in \N_0
$
almost surely for some deterministic $\theta$.
Then, $D = X' - X\in\Z$, which is crucial
to obtain \eqref{crux1} in the derivation 
of the above Chen--Stein bound. 
In this case, the factor $D^2(D^2-1)$ in  $(\spadesuit)$
satisfies 
\noi
\begin{align*}
 D^2(D^2-1)=0
 \quad\Longleftrightarrow\quad
 D\in\{-1,0,1\}.
\end{align*}
Hence it does not penalize a pair that stays put or makes a unit jump,
but it is strictly positive whenever the pair makes a jump of size at
least two.  The normalization by $a$ extracts the first-order rate of
such non-unit jumps.  More precisely, for the pair
$(F+\theta,F_t+\theta)$ associated with an order-$r$ chaos, one has
$a=1-e^{-rt}$ and $D_t=F_t-F$, and Proposition~\ref{prop_J} identifies
\noi
\begin{align*}
 \lim_{t\downarrow0}
 \frac{\E[D_t^2(D_t^2-1)]}
 {6(1-e^{-rt})}
 =\frac{\cJ(F)}{6}.
\end{align*}

At the level of the underlying Poisson configuration, the coupling
$\eta\mapsto\eta_t$ in  \eqref{P_pair} is generated by the Poisson
immigration--death semigroup: existing points are independently
removed, while an independent Poisson cloud is added.  The increment
$D_t$, however, records the resulting change of the functional $F$,
which may exceed one even when the configuration undergoes only one
elementary birth or death.  
Thus, the limiting quantity  $\cJ(F)$ measures the infinitesimal defect of the induced
integer-valued dynamics from the one-unit jump structure required by
the Poisson target.
\end{remark}

Under the same lattice condition \eqref{cond_La}, 
the spectral representation \eqref{J_sp} of $\cJ(F)$ in
Proposition~\ref{prop_J} allows $\cJ(F)$ to be controlled by the
four-moment defect.  In the notation of Theorem~\ref{thm_master}, with
$\mathfrak d(F)=|\cD(F)|+\rho(F)$, one obtains

\noi
\begin{align*}
 \cJ(F)
 \leq4\sqrt{m_2(F)\mathfrak d(F)}
 +12\mathfrak d(F).
\end{align*}
For a pure Poisson-chaos variable, $\rho(F)=0$ and
$\cD(F)=\cP(F)\geq0$, so that

\noi
\begin{align*}
 \cJ(F)
 \leq4\sqrt{m_2(F)\cD(F)}
 +8\cD(F).
\end{align*}

\noi
In the Rademacher setting, $\rho(F)$ is precisely the
collision correction controlled by the maximal influence
(see Corollary \ref{cor_R_quant} and Lemma \ref{lem_R_top}).
Therefore,
up to this necessary collision term, the same moment defect $\cR(F)$ that
appears in the four-moment criterion also suppresses non-unit jumps.
Note that for $H = F+B$ with $F\in\cC^{\mathsf{S}}_q$ dominant pure chaos and $B$ of 
smaller  order, the same argument can go through for bounding 
the quantity $\cJ_q(F+B)$ by the moment defect of $F$ and the $L^4$-size 
of lower-order chaoses; see \eqref{Jq_cont} and \eqref{Jq_pure}.

%

%
%

Now we turn to   the conditional quadratic-variation term $(\clubsuit)$ 
in the Chen--Stein bound \eqref{CS_bdd}.  For $X=F+\theta$, $D=D_t$, and
$a=1-e^{-rt}$, the estimate \eqref{cond_master} yields
\noi
\begin{align*}
 &\limsup_{t\downarrow0}
 \E\Big|  \frac{\E[D_t^2 |  F]}{2(1-e^{-rt})}   -\frac{F+2\theta}{2}  \Big| 
 \leq  | m_2(F) - \theta| +  \|\cT_F\|_2,
\end{align*}
where
$\cT_F=r^{-1}\Gamma(F,F)-m_2(F)-\frac{F}{2}$.
Thus the first Chen--Stein term $(\clubsuit)$   naturally reduces to bounding the
variance mismatch $|m_2(F)-\theta|$ and the spectral remainder
$\|\cT_F\|_2$.  The estimate \eqref{T_master} controls the latter by
the square root of the same defect $\mathfrak{d}(F)$.  Combining this estimate with the
control of $\cJ(F)$ gives the master total-variation bound
\eqref{master_TV}.

We now describe the common spectral source of these estimates.  Let
$F$ belong to a chaos of order $r$ and write

\noi
\begin{align}
 F^2=\sum_{q=0}^{2r}J_q(F^2),
 \label{2_intro}
\end{align}
where $J_q$ denotes the orthogonal projection onto the $q$-th chaos. 
Let $\cD(F)$ denote $\cP(F)$ on Poisson space and $\cR(F)$ on
Rademacher space.  Orthogonality gives the exact identity
(see \eqref{D_ide}-\eqref{E_abs})

\noi
\begin{align}
 \cD(F)
 &=\|J_{2r}(F^2)\|_2^2-2m_2(F)^2
 +\sum_{q=1}^{2r-1}\|J_q(F^2)\|_2^2
 \ind_{\{q\neq r\}}
 +\|J_r(F^2)-F\|_2^2.
 \label{CD_intro}
\end{align}
On Poisson space, the top projection satisfies
$\|J_{2r}(F^2)\|_2^2\geq 2m_2(F)^2$;
see Lemma \ref{lem_P_top}.
 On Rademacher space, the same
lower bound holds up to an explicit collision term controlled by the
maximal influence;
see Lemma \ref{lem_R_top}.
Hence a small moment defect, together with a small collision
correction in the Rademacher case, forces all the nonnegative
lower-order spectral errors in \eqref{CD_intro} to collapse.  This
collapse simultaneously controls $m_3(F)-m_2(F)$,
$\cT_F$, and  $\cJ(F)$.  Section~\ref{SEC3}
turns this mechanism into the master estimates of
Theorem~\ref{thm_master} and then specializes them to the Poisson and
Rademacher settings;
see Corollary \ref{cor_P_quant}
and  Corollary \ref{cor_R_quant}.

Theorem~\ref{thm_pert} is useful precisely when the leading term
$F_n$ does not satisfy the lattice condition \eqref{cond_La}.  The
lattice condition is imposed instead on the full variable $X_n$, and
the regression remainder is constructed from the lower-order
projections; see $\widetilde R_t$ in \eqref{pert_wtR}.  We now give a
Poisson-space application in every prescribed chaos order.

\begin{corollary}
\label{cor_hyper}

Fix $q\geq2$.  For each $n$, let
$\big(N_{n,i,j}:1\leq i\leq m_n,  1\leq j\leq q\big)$
be independent random variables with $N_{n,i,j} \sim \Pois(\alpha_{n,i,j})$.  
Put

\noi
\begin{align}
 \beta_{n,i}
 :=\prod_{j=1}^q\alpha_{n,i,j},
 \quad
 X_n
 :=\sum_{i=1}^{m_n}
 \prod_{j=1}^qN_{n,i,j},
\AND
 \delta_n
 :=
 \max_{\substack{1\leq i\leq m_n\\1\leq j\leq q}}
 \alpha_{n,i,j}.
 \label{hyper_count}
\end{align}
If
\noi
\begin{align}
 \delta_n\to0
 \AND
 \sum_{i=1}^{m_n}\beta_{n,i}\to\lambda>0,
\notag 
\end{align}
then $X_n$ converges to $\Pois(\lambda)$ in total variation, and its
first four moments converge to those of the limit.
 
Moreover, there exists a constant $C_{q,\lambda}$ depending only on $q$ and $\lambda$
 such that

\noi
\begin{align}
 \dTV\bigl( X_n, P_\lambda\bigr)
 \leq C_{q,\lambda} 
 \left(
 \left|\sum_{i=1}^{m_n}\beta_{n,i}-\lambda\right|
 +\sqrt{\delta_n}\right)
 \label{hyper_rate}
\end{align}
for all
sufficiently large $n$.
\end{corollary}

We also record two elementary consequences.  The first is the
triangular-array form of the law of small numbers.

\begin{corollary}[\textsf{Law of small numbers}]
\label{cor_small}
Let $(B_{n,k})_{1\leq k\leq N_n}$ be independent 
$\{0,1\}$-valued Bernoulli variables
with success probabilities $p_{n,k}$, and put

\noi
\begin{align}
 X_n:=\sum_{k=1}^{N_n}B_{n,k}
 \AND
 \theta_n:= \E[X_n] =\sum_{k=1}^{N_n}p_{n,k}.
\notag
\end{align}
If

\noi
\begin{align}
 \theta_n\to\lambda>0
\AND
 \max_{1\leq k\leq N_n}p_{n,k}\to0,
\notag
\end{align}
then $X_n$ converges to $\Pois(\lambda)$ in total variation and its
first four moments converge to those of the limit.
\end{corollary}

The second consequence records stability under independent
superposition.

\begin{corollary}[\textsf{Law of small chaotic blocks}]
\label{cor_blocks}
Fix an order $r\geq1$.  For each $n$, let
$F_{n,1}$, $\ldots$, $F_{n,K_n}$ be independent centered variables in
order-$r$ chaoses
over mutually independent component spaces, 
and
suppose that

\noi
\begin{align}
 X_{n,k}:=F_{n,k}+\theta_{n,k}\in\N_0.
\notag
\end{align}
Put
$F_n:=\sum_{k=1}^{K_n} F_{n,k}$,
$\theta_n:=\sum_{k=1}^{K_n}\theta_{n,k}$, and
$X_n:=F_n+\theta_n$.
In the Poisson setting, assume

\noi
\begin{align}
 \theta_n\to\lambda,
 \quad
 \sum_km_2(F_{n,k})\to\lambda,
 \AND
 \sum_k\cP(F_{n,k})\to0.
\notag
\end{align}
Then $X_n$ converges to $\Pois(\lambda)$ in total variation.  In the
Rademacher setting, the same conclusion holds with $\cP$ replaced by
$\cR$, provided that the maximal coordinate influence $\cM(F_n)$ of $F_n$ tends
to zero.
\end{corollary}

\medskip
\noi
$\bul$ {\bf Organization of the paper.}
Section~\ref{SEC2} collects the preliminary material.  More precisely,
Section~\ref{SEC_21} records the moment identities, the basic
total-variation estimate, and the uniform-integrability criterion used
throughout.  Section~\ref{SEC_24} establishes the exchangeable-pair
Chen--Stein bound, while Sections~\ref{SEC_22} and \ref{SEC_23}
construct the Ornstein--Uhlenbeck exchangeable pairs on Poisson and
generalized Rademacher spaces, respectively.

Section~\ref{SEC3} contains the common proof mechanism and its
applications.  Section~\ref{SEC_31} develops the spectral-collapse and
unit-jump estimates, derives the quantitative pure-chaos Poisson and
Rademacher bounds, and proves Theorems~\ref{thm_P} and \ref{thm_R}.
Section~\ref{SEC_33} proves the quantitative dominant-chaos bound
in Theorem~\ref{thm_pert} simultaneously on the two spaces, whereas
Section~\ref{SEC_34} contains a few applications.

Section~\ref{SEC4} presents further examples and rigidity results.
Section~\ref{SEC_41} gives the exact quadratic counterexample.
Section~\ref{SEC_42} explains why two moments do not characterize a
Poisson limit and why uniform integrability is needed for the converse
four-moment implication.  Section~\ref{SEC_43} describes exact Poisson
laws in the first Poisson chaos, proves rigidity in higher Poisson
chaoses, and then gives two higher-order sequences with vanishing
moment defect and centered Poisson weak limits.  It also rules out
exact Poisson laws in the first Rademacher chaos.  Finally,
Section~\ref{SEC_44} proves the bounded-bias coordinate-rigidity
theorem (Theorem \ref{thm_bias}).

\section{Preliminaries}
\label{SEC2}

Let us fix some more notations. 

\medskip
\noi
{\bf $\bul$  Notations.}
All random variables are defined on a probability space
$(\Omega,\mathcal F,\PP)$.  For $p\in[1,\infty)$ and
$U\in L^p(\PP)$, we write
$\|U\|_p$ for the usual $L^p(\PP)$-norm of a real-valued 
random variable $U$. 
 If
$g:\N_0\to\R$, then

\noi
\begin{align} \label{sup_g}
 \norm{g}_\infty
 :=
 \sup\big\{ |g(j)|:  j\in\N_0 \big\}.
\end{align}
We write $\ind_A$ for the indicator of a set or event $A$,
$A^c$ for its complement, and $|A|$ for its cardinality.
The symbol $\delta_z$ denotes the Dirac measure at $z$.
For a measure $\mu$, we put
$\mu^q:=\mu^{\otimes q}$,
and let $L_s^2(\mu^q)$ denote the symmetric subspace of
$L^2(\mu^q)$.  For kernels $f$ and $g$, the symbols
$f\ot g$, $f\sot g$, and $f\ot_s g$ denote, respectively, their
tensor product, canonical symmetrization, and $s$-th contraction.
When working on either the Poisson or the Rademacher space,
$J_q$ denotes the orthogonal projection onto the $q$-th chaos.

%
%
%
%

\subsection{Moments, total variation, and uniform integrability}
\label{SEC_21}

Recall that $P_\lambda$ denotes the $\text{Poisson}(\lambda)$ distribution
and recall the definition \eqref{def_TV} of total variation distance. 
Now we record a useful inequality:

\noi
\begin{align}
 \dTV (P_a, P_b )
 \leq1-e^{-\abs{a-b}}
 \leq\abs{a-b}.
 \label{Poi_param}
\end{align}
The above inequality can be proved by a standard monotone coupling argument
based on the fact that the sum of two independent Poisson random variables
is again Poisson.

The following elementary moment relation  will be used repeatedly.

\begin{lemma}
\label{lem_raw}
Suppose $(F_n)_{n\geq 1}$ is a sequence of centered random variables
and $(\theta_n)_{n\geq 1}$ is a deterministic sequence of real numbers. 
Let $X_n=F_n+\theta_n$. Then,  convergence of the first
four   moments of $X_n$ to those of $P_\lambda$ is equivalent
to
\[
(\theta_n, m_2(F_n),  m_3(F_n), m_4(F_n)   )
\to (\lambda, \lambda, \lambda, 3\lambda^2+\lambda).
\]

\end{lemma}

\begin{proof}
For $X = F+\theta$,
one can easily check the following  identities:
\begin{align}
 \E[X]=\theta,
\quad
 \E[X^2]&=\theta^2+m_2(F),
\quad
 \E[X^3]=\theta^3+3\theta m_2(F)+m_3(F),
 \notag\\
 \E[X^4]&=\theta^4+6\theta^2m_2(F)
  +4\theta m_3(F)+m_4(F),
 \label{raw_id}
\end{align}
which allow one to compare the moments successively with
\eqref{4M}.  This proves the lemma.
\end{proof}

We next record an elementary fact from \cite[Theorem 3.6]{Bil99}.

\begin{lemma}
\label{lem_UI}
Let $Y_n\geq0$ and $Y\geq0$ have finite expectations.  If
$Y_n\xrightarrow{\mathrm{law}}Y$
and
$
 \E[Y_n]\to\E[Y]$,
then $(Y_n)_{n\geq1}$ is uniformly integrable.
\end{lemma}

\subsection{A Chen--Stein bound via exchangeable pairs}
\label{SEC_24}

We say $(X, X')$ is   an exchangeable pair
if $(X, X')$ has the same law as $(X', X)$.
The following lemma is a regression-based variant of the classical
exchangeable-pair approach to Poisson approximation; see
\cite{CDM05} and \cite[Section~4.4]{Ros11}.

\begin{lemma}
\label{lem_Stein}
Let $(X,X')$ be an exchangeable pair with values in $\N_0$
and $X\in L^4(\PP)$, 
and define 
$D:=X'-X$.
 Suppose that, for some
$a>0$, $\theta>0$ and a random variable $R\in L^1(\Omega, \sigma\{X\}, \PP)$,
we have

\noi
\begin{align}
 \E[D |  X]=-a(X-\theta)+R.
 \label{lin_reg}
\end{align}
Set $c_\theta:= (1-e^{-\theta})/\theta.$
Then,
\begin{align}
 &\dTV ( X, P_\theta  )
\leq c_\theta
 \bigg\{
 \E\Big|
 \frac{\E[D^2 |  X]}{2a}
 -\frac{X+\theta}{2}
 \Big|
 +\frac{\E[D^2(D^2-1)]}{6a}
 \bigg\}
 +\left(1+\frac{c_\theta}{2}\right)
  \frac{\E|R|}{a}.
 \label{Stein_pair}
\end{align}
\end{lemma}

\begin{proof}
For $A\subseteq\N_0$, let $f_A$ solve the Chen--Stein equation
(with the convention $f_A(0)=0$)

\noi
\begin{align}
 \theta f_A(k+1)-k f_A(k)
 =\ind_A(k)-\PP(P_\theta\in A).
\notag
\end{align}

\noi
Then, we have, with $ \Delta f(k):=f(k+1)-f(k)$,

\noi
\begin{align}
 \norm{f_A}_\infty\leq1
 \AND
 \norm{\Delta f_A}_\infty\leq c_\theta;
 \label{Stein_factor}
\end{align}
See \cite[Lemma 1.1.1]{BHJ}. Here, $\| g \|_\infty$
is the sup-norm of $g$ as in \eqref{sup_g}.

By  exchangeability, it holds  
that
$
\E[H(X,X')]=\E[H(X',X)],
$
whenever $H:\N_0^2\to\R$ satisfies $H(X, X')\in L^1(\PP)$. 
Applying this identity, 
we can obtain

\noi
\begin{align}
&\E[D f(X)]
+\frac{1}{2}\E[D(D+1)\Delta f(X)]
=\E[T_f(X,D)],
\label{sym_rem}
\end{align}

\noi
where
\begin{align}
T_f(x,d)
&:=
\frac{d}{2}\big(  f(x)-f(x+d) \big)
+\frac{d(d+1)}{4}\Delta f(x)
+\frac{d(d-1)}{4}\Delta f(x+d).
\label{Rf}
\end{align}
Notice that
\begin{align}\label{xdd1}
T_f(x,d)=T_f(x+d,-d),
\end{align}
and
\begin{align}\label{xdd2}
T_f(x,d)=0 
\quad\text{for $d\in \{-1,0,1\}.$}
\end{align}

\noi
Using  \eqref{Rf}
and  the telescoping  sums
\[
f(x+m) - f(x) = \sum_{j=0}^{m-1}\Delta f(x+j), 
\]
we obtain that for $m\geq 2$ and $x\in\N_0$,
\noi
\begin{align} 
T_f(x,m)
&= \frac{m}{4} \bigg[ (m-1) \Big( \Delta f(x)+\Delta f(x+m)  \Big) 
  -   2\sum_{j=1}^{m-1}\Delta f(x+j)   \bigg],
\end{align}
and consequently,
$  |T_f(x,m)|  \leq
 m(m-1) \|\Delta f\|_\infty \leq 
 \frac{m^2(m^2-1)}{6}\|\Delta f\|_\infty.$
The case $d\leq -2$ follows from
\eqref{xdd1}.
Thus, together with \eqref{xdd2},
we can deduce 
$
|T_f(x,d)|
\leq
 \|\Delta f\|_\infty 
d^2(d^2-1)/6
$
for $x\in\N_0$ with $x+d\in\N_0$.

Since $X, X'\in\N_0$   and $D = X' - X$, we have 
\begin{align} \label{crux1}
|T_f(X, D)| \leq  \frac{D^2(D^2-1)}{6} \| \Delta f\|_\infty.
\end{align}

\noi
Then, it follows from \eqref{sym_rem} and \eqref{crux1}
that 

\noi
\begin{align}
\left|
\E[D f(X)]
+\frac{1}{2}\E[D(D+1)\Delta f(X)]
\right|
\leq
\frac{\|\Delta f\|_\infty}{6}
\E[D^2(D^2-1)].
\notag
\end{align}

\noi
 Now using the approximate linear regression condition
\eqref{lin_reg},
we have

\noi
\begin{align*}
& \E\big[ \theta f_A(X+1)-Xf_A(X) \big]
=
 \theta\E[\Delta f_A(X)]  +\frac1a\E[D f_A(X)]  -\frac1a\E[Rf_A(X)] \\
& = \frac{1}{a} \E\Big[ D f_A(X)  +\frac{1}{2} D(D+1)\Delta f_A(X) \Big]
+ \E\Big[ \Big( \theta - \frac{\E[ D(D+1) |X ]}{2a} \Big) \Delta f_A(X) \Big] - \frac{1}{a} \E[ R f_A(X) ].
\end{align*}
For the middle term, we use \eqref{lin_reg} again to rewrite 

\noi
\begin{align}
 \theta-\frac{\E[D(D+1) |  X]}{2a}
 =\frac{X+\theta}{2}
  -\frac{\E[D^2 |  X]}{2a}
  -\frac{R}{2a}.
\notag
\end{align}
Therefore using \eqref{Stein_factor}, 
we can bound $ \E [ \theta f_A(X+1)-Xf_A(X)  ]$
by the right-hand side of \eqref{Stein_pair}.
Taking the supremum over $A\subset\N_0$ concludes the proof. 
\qedhere

\end{proof}

 In the next two subsections, we will recall the exchangeable pair constructions
 from \cite{DVZ18} and \cite{Zhe19}.

\subsection{Exchangeable pairs on Poisson chaos via thinning}
\label{SEC_22}

Let $\eta$ be a Poisson random measure on $(\cZ,\sZ)$ with
$\sigma$-finite control measure $\mu$, and let $\widehat\eta=\eta-\mu$
be the corresponding compensated Poisson random measure. 
For $q\in\N$, write $I_q^{\eta}(f)$ for the $q$-th multiple
Wiener--It\^o integral of a symmetric kernel
$f\in L_s^2(\mu^q)$.
  The $q$-th Poisson chaos is
\noi
\begin{align}
\qquad\qquad
 \cC_q^{ \mathsf{P}}
 :=\{I_q^{\eta}(f):f\in L_s^2(\mu^q)\}
 \quad
 \text{with the convention $\cC_0^{ \mathsf{P}}:=\R.$}
\notag\end{align}
Every $H\in L^2(\sigma\{\eta\})$ has the unique chaos expansion in $L^2(\PP)$:

\noi
\begin{align}
 H=\E[H]+\sum_{q\geq1}J_q H, 
\notag\end{align}

\noi
where   $J_q$ denotes the orthogonal projection operator onto $\cC^\eta_q$;
see, e.g., \cite[Theorem 18.10]{LP18}.
 The Poisson Ornstein--Uhlenbeck generator and semigroup are
 given by

\noi
\begin{align}
 LH=-\sum_{q\geq1} qJ_q H
 \AND 
 P_tH&=\sum_{q\geq0}e^{-qt}J_qH.
 \label{P_OU}
\end{align}
Whenever the expressions are defined, the carr\'e-du-champ operator $\Gamma$ is
defined by 
\noi
\begin{align}
 \Gamma(G,H)
 :=\frac12\bigl(L(GH)-G LH-H LG\bigr).
 \label{Gamma_def}
\end{align}
If $F\in\cC_q^{ \mathsf{P}}\cap L^4(\PP)$, then $F^2$ belongs to
$\bigoplus_{j=0}^{2q}\cC_j^{ \mathsf{P}}$; see
\cite[Lemma~2.4]{DP18a} and see also
\cite[Proposition 1.10]{Zhe26}.  Thus all the spectral quantities used below
are well defined under the sole fourth-moment assumption.
For example, for $F\in\cC_q^{ \mathsf{P}}\cap L^4(\PP)$,
we have 
\begin{align}\label{rel_Gamma}
\Gamma(F, F) = q m_2(F) + \sum_{k=1}^{2q-1} (q-\frac{k}{2}) J_k(F^2).
\end{align}

We use exactly the exchangeable pair construction in
\cite[Section~3.1]{DVZ18}.  Since we are only concerned with
distributional properties in this paper,
 we can assume (without loss of   generality) that $\eta$ is a proper point process,  meaning
that it can be represented as
$
\eta = \sum_{n=1}^\kappa \delta_{X_n},
$
where $(X_n)_{n\geq 1}$ are random variables with values in $\cZ$ and 
$\kappa$ is a $\N_0\cup\{+\infty\}$-valued random variable; see, e.g., \cite[Corollary 3.7]{LP18}
for this representation.  Let $(Y_n)_{n\geq 1}$ be a sequence of i.i.d. exponential random variables
with unit mean, independent from $(\kappa, X_n)$.
Then, the marked point process 
$
\xi = \sum_{n=1}^\kappa \delta_{(X_n, Y_n)}
$
is a Poisson point process with intensity measure $\mu\otimes\mathbb{Q}$, with $\mathbb{Q}$ 
the distribution measure of $Y_1$. Define $\eta_{e^{-t}}(A) := \xi( A \times [t,+\infty))$ for $A\in\sZ$.
Then, $\eta_{e^{-t}}$ is a Poisson point process on $\cZ$ with control measure $e^{-t}\mu$.
This is exactly the $e^{-t}$-thinning of $\eta$. Now   let
$\eta'_{1-e^{-t}}$ be an independent Poisson process with control measure
$(1-e^{-t})\mu$, and put
\noi
\begin{align}
 \eta^t:=\eta_{e^{-t}}+\eta'_{1-e^{-t}}.
 \label{P_pair}
\end{align}
Then $(\eta,\eta^t)$ is exchangeable (Lemma 3.1 in \cite{DVZ18}) and, for
$F=I_q^{\eta}(f)$, we set
\noi
\begin{align}
 F_t:=I_q^{\eta^t}(f),
 \label{Ft_P}
\end{align}
which forms an exchangeable pair with $F$.
The next result is   taken from  \cite{DVZ18}.

\begin{proposition}[\textsf{Proposition~3.2 in \cite{DVZ18}}]
\label{prop_P_pair}
Let $F=I_q^{\eta}(f)\in L^4(\PP)$ and define $F_t$ by
\eqref{Ft_P}.  Then $(F,F_t)$ is exchangeable and

\noi
\begin{align}
 \frac1t\E[F_t-F | \eta]  &\xrightarrow{t\downarrow 0} -qF  \quad\text{in $L^4(\PP)$},  
 \notag \\
 \frac1t\E[(F_t-F)^2 | \eta]  &\xrightarrow{t\downarrow 0}    2\Gamma(F,F)  \quad \text{in $L^2(\PP)$},  
 \notag \\
 \frac1t\E[(F_t-F)^4] &\xrightarrow{t\downarrow 0}   -4q\,m_4(F)+12\E[F^2\Gamma(F,F)]  \geq0.  
 \notag
\end{align}

\noi
Moreover, the first regression is exact:
$
 \E[F_t-F | \eta]
 =-(1-e^{-qt})F.
$
Moreover, if $F, G$ live in the sum of finitely many chaoses, 
with $F_t, G_t$ defined as above, we have 
\noi
\begin{align}
 \frac1t\E[(F_t-F)(G_t-G) | \eta]
\xrightarrow{t\downarrow 0}    2\Gamma(F,G)  \quad\text{in $L^2(\PP)$}.
 \label{P_cross_reg}
\end{align}
\end{proposition}

Note that \eqref{P_cross_reg} is stated in \cite[Proposition~3.2]{DVZ18}
only for $F\in\cC_q^{ \mathsf{P}}$ and $G\in\cC_p^{ \mathsf{P}}$, but
its proof works for $F, G$ in the sum of finitely many chaoses.

\subsection{Exchangeable pairs on  Rademacher chaos via coordinate refreshing}
\label{SEC_23}

Let $(X_k)_{k\geq1}$ be independent Rademacher variables with
\noi
\begin{align}
 \PP(X_k=1)=p_k \in(0,1),
\qquad
 \PP(X_k=-1)=q_k:=1-p_k,
\notag
\end{align}
and let $Y_k$ be defined as in \eqref{Rad_norm}. 
Starting from the chaos decomposition (\cite[Section 6]{Pri08})
\[
L^2(\sigma\{X_k:k\geq1\}) = \bigoplus_{q\in\N_0} \cC^{\mathsf{R}}_q,
\]
with $\cC^{\mathsf{R}}_q = \{ Q_q(f; Y): f\in\fH^{\odot q}_0\}$ as in \eqref{Rad_chaos},
one can define the Rademacher Ornstein--Uhlenbeck generator $L$ and semigroup 
$P_t$ in the same way as in \eqref{P_OU}, with $J_q$ now denoting the orthogonal 
projection onto $\cC^{\mathsf{R}}_q$. 
One can also define the carr\'e-du-champ operator 
$\Gamma$ by \eqref{Gamma_def}.
If $F\in\cC_d^{\mathsf{R}}\cap L^4(\PP)$, then $F^2$ belongs to
$\bigoplus_{j=0}^{2d}\cC_j^{\mathsf{R}}$.  Indeed, the assertion is
immediate when the kernel of $F$ has finite support; the general case follows
from 
conditional-expectation truncations (i.e., martingale approximation)
and the fact that $\bigoplus_{j=0}^{2d}\cC_j^{\mathsf{R}}$   is closed in $L^2(\PP)$.
Thus all the spectral quantities used below
are well defined under the sole fourth-moment assumption.

For a functional $H$, let $H_k^+$ and $H_k^-$ denote the versions
obtained by fixing $X_k$ equal to $1$ and $-1$, respectively.  The
discrete Malliavin gradient is

\noi
\begin{align}
 D_kH
 :=\sqrt{p_kq_k}\,(H_k^+-H_k^-).
\label{Dk_def}
\end{align}
Thus, with

\noi
\begin{align}
 z_k:=H_k^+-H_k^-,
 \label{zk_def}
\end{align}
we have the precise relation

\noi
\begin{align}
 D_kH=\sqrt{p_kq_k}\,z_k.
\label{Dk_rel}
\end{align}
For $F\in\cC_d^{\mathsf{R}}$, with influence $\Inf_k$ as in \eqref{Rad_inf},
we can easily verify that

\noi
\begin{align}
d d! \Inf_k(F)=\E[(D_kF)^2]
\AND 
d! \sum_{k\geq1}\Inf_k(F)= m_2(F).
\label{D_inf}
\end{align}
If $F+\theta$ is integer-valued, then Fubini's theorem and
$p_kq_k>0$ imply that $F_k^++\theta$ and $F_k^-+\theta$ are both
integer-valued for almost every configuration of the other
coordinates.  Consequently, $z_k\in\Z$ almost surely.
See Remark \ref{rem_32}-(ii) for another appearance of $z_k$.

\medskip

The following construction is from \cite{Zhe19}.
Let $  \cG:=\sigma\{X_k:k\geq1\}.$
Let $(X_k')_{k\geq1}$ be an independent copy of $(X_k)$ and let
$(\tau_k)_{k\geq1}$ be independent standard exponential variables,
independent of both sequences.  Set
\noi
\begin{align}
 X_k^t
 :=X_k\ind_{\{\tau_k\geq t\}}
   +X_k'\ind_{\{\tau_k<t\}},
 \qquad t\geq0,
\notag
\end{align}
and let $Y^t$ be its normalized version.  If
$F=Q_d(f;Y)$, put
$
 F_t:=Q_d(f;Y^t).
 $

\begin{proposition}[\textsf{Lemma 2.2 and Proposition 2.3 in \cite{Zhe19}}]
\label{prop_R_pair}
Let $F=Q_d(f;Y)\in L^4(\PP)$ and define $F_t$ as above.
 Then $(F,F_t)$ is exchangeable and

\noi
\begin{align}
 \frac1t\E[F_t-F | \cG]
 &\xrightarrow{t\downarrow 0}   -dF   \quad  \text{in $L^4(\PP)$}
\notag
 \\
 \frac1t\E[(F_t-F)^2 | \cG]
  &\xrightarrow{t\downarrow 0}     2\Gamma(F,F)
\quad  \text{in $L^2(\PP)$},
\notag
  \\
 \frac1t\E[(F_t-F)^4]
 &\xrightarrow{t\downarrow 0}      -4d\,m_4(F)+12\E[F^2\Gamma(F,F)]
 \geq  0.
\notag
\end{align}

\noi
Moreover, $\E[F_t-F | \cG]=-(1-e^{-dt})F$.
\end{proposition}

The same semigroup argument also gives the cross-regression relation
needed below.  If $U,V$ live in the sum of finitely many Rademacher
chaoses and $U_t,V_t$ are defined using the same refreshed sequence,
then

\noi
\begin{align}
 \frac1t\E[(U_t-U)(V_t-V) | \cG]
 \xrightarrow{t\downarrow0}2\Gamma(U,V)
 \quad\text{in }L^2(\PP).
 \label{R_cross_reg}
\end{align}
Indeed, conditional on $\cG$, the Markov property gives
\[
 \E[(U_t-U)(V_t-V) | \cG]
 =P_t(UV)-UP_tV-VP_tU+UV.
\]
Since $U,V$, and $UV$ live in finite sums of chaoses, division by $t$
and the definition of the generator yield \eqref{R_cross_reg}.

\section{A unified  spectral proof and applications}
\label{SEC3}

\subsection{Spectral bounds and unit jumps}
\label{SEC_31}

In this subsection,  $F\in L^4(\PP)$  belongs to a fixed chaos of order $r$ in
the Poisson or Rademacher setting.
Let $J_q$ denote the orthogonal projection operator onto the $q$-th chaos.
Then, from the discussions in Sections \ref{SEC_22}-\ref{SEC_23}, 
we have the chaos decomposition \eqref{2_intro}:
$
 F^2=\sum_{q=0}^{2r} J_q(F^2)
$.
By orthogonality, we   get

\noi
\begin{align}
 J_0(F^2)
 =m_2(F),
 \quad
m_4(F)=\sum_{q=0}^{2r} \| J_q(F^2)\|_2^2,
\quad
 m_3(F)= \E[F J_r(F^2)].
 \label{sp_1}
\end{align}
and
\begin{align}\label{sp_2}
\| J_r(F^2) -F\|_2^2= \|J_r(F^2)\|_2^2  + m_2(F) - 2m_3(F).
\end{align}

\noi
When no underlying space is specified, we write\footnote{That is, 
 $\cD=\cP$ on Poisson space as in  \eqref{def_P}
  and $\cD=\cR$ on Rademacher space as in \eqref{def_R}.}

\noi
\begin{align}
 \cD(F)
 :=m_4(F)-3m_2(F)^2-2m_3(F)+m_2(F).
 \label{def_CD}
\end{align}

\noi
Then,  a direct calculation with \eqref{sp_1}
yields the basic identity

\noi
\begin{align}  \label{D_ide}
\begin{aligned}
 \cD(F)
 &= \bigg(\sum_{q=1}^{2r-1} \| J_q(F^2)\|_2^2 \ind_{\{ q\neq r\}} \bigg)
 + m_2(F)^2 + \| J_r(F^2)\|_2^2 + \| J_{2r}(F^2)\|_2^2 \\
 &\qquad\qquad   -3m_2(F)^2  -2 \E[ F J_r(F^2)] +m_2(F). \\
 &= \|J_{2r}(F^2)\|_2^2 -  2m_2(F)^2 + \fE_r(F),
\end{aligned}
\end{align}
where

\noi
\begin{align}
 \fE_r(F)
 :=\sum_{q=1}^{2r-1} \| J_q(F^2)\|_2^2 \, \ind_{\{q\ne r\}}
   + \| J_r(F^2) -F\|_2^2. 
 \label{E_abs}
\end{align}

\begin{proposition}
\label{prop_J}
Fix $r\in\N$
and
let $(F_t, F)$ be the exchangeable pair of
Proposition~\ref{prop_P_pair} with $q=r$ or Proposition~\ref{prop_R_pair}
with $d=r$.
Assume $F$ lives in the $r$-th chaos 
and 
put
$D_t:=F_t-F.$
Then,

\noi
\begin{align}
 \cJ(F)
 :=\lim_{t\downarrow0}
 \frac{\E[D_t^2(D_t^2-1)]}{1-e^{-rt}}
 \label{J_limit}
\end{align}
exists  and

\noi
\begin{align}
 \cJ(F)
 =8m_4(F)-\frac{6}{r}  \sum_{q=1}^{2r}  q\|J_q(F^2)\|_2^2
 -2m_2(F).
 \label{J_sp}
\end{align}
If in addition $F$  satisfies the lattice condition
\eqref{cond_La}, then $\cJ(F)\geq 0$.

\end{proposition}

\begin{proof}
By the construction of exchangeable pairs  
through  Ornstein-Uhlenbeck semigroup, 
we have 

\noi
\begin{align} \label{m2Dt}
\begin{aligned} 
\E[ D_t^2] &= \E[ F_t^2] + \E[F^2] - 2\E[ F F_t] = 2\E[ F^2] - 2 \E[ F P_tF]  \\
& =2(1-e^{-rt})   m_2(F)
\end{aligned}
\end{align}
and

\noi
\begin{align} \label{m4Dt}
\begin{aligned} 
 \E[D_t^4]
 &= \E[ F_t^4 + F^4 - 4F^3_t F - 4F^3F_t + 6F^2 F_t^2] \\
 & =  2\E[F^4] - 8 \E[ F^3 P_tF] + 6 \E[F^2 P_t(F^2)] \\
 &= 2m_4(F)-8e^{-rt}m_4(F)
  +6\sum_{q=0}^{2r}e^{-qt } \|J_q(F^2)\|_2^2 \\
  &= 8(1-e^{-rt} ) m_4(F) - 6 \sum_{q=0}^{2r} (1 - e^{-qt} ) \|J_q(F^2)\|_2^2 
\end{aligned}
\end{align}

\noi
Then, the desired limiting expression \eqref{J_limit}
follows from the above two moment expressions
with $(1-e^{-qt})\sim qt$ as $t \downarrow 0$.
Lastly, if $F+\theta\in\N_0$ for some $\theta$,
then
 $F_t+\theta\in\N_0$ as well,
so that $D_t\in\Z$ and $D_t^2(D_t^2-1)\geq0$.  
Hence $\cJ(F)\geq 0$.
\qedhere

\end{proof}

\begin{remark} \label{rem_32}
%
%

Suppose $F = f( X_1, \ldots, X_m)$ depends on finitely many 
Rademacher random variables
(so that $F$ is uniformly bounded almost surely)
and define 
$F_t= f(  X^t_1, \ldots, X^t_m)$.
Then, 
\begin{align*}
(F_t - F)^2 
& =
 \sum_{k=1}^m  | f(  X^t_1, \ldots, X^t_m) - f(X_1, \ldots, X_m)|^2 \ind_{\mathcal{E}_k} 
 +    (F_t  - F)^2 \ind_{\mathcal{E}_0},
\end{align*}
where $\mathcal{E}_k:=\{\tau_k<t,\ \tau_j\geq t\text{ for all }j\ne k\}$
 so that $\PP(\mathcal{E}_k) =   (1- e^{-t})  e^{-(m -1) t} $,
and  $\mathcal{E}_0$ denotes the event that
there are at least two different coordinates  so that 
$\PP(\mathcal{E}_0) = O(t^2)$,
while we have $F_t = F$  if no coordinate is refreshed before time $t$.
Then, using the independence and $z_k = F_k^+ - F_k^-$
as in \eqref{zk_def},
we get 
\begin{align*}
\E[ (F_t - F)^2 ]
& =
 \sum_{k=1}^m 2 p_kq_k \E[ z_k^2]  (1- e^{-t})  e^{-(m -1) t} 
 + O (t^2)
\end{align*}
and in the same way, we get 
\begin{align*}
\E[ (F_t - F)^4 ]
& =
 \sum_{k=1}^m 2 p_kq_k \E[ z_k^4]  (1- e^{-t})  e^{-(m -1) t} 
 + O (t^2)
\end{align*}
This implies 
\begin{align}\label{J_R_coord}
\frac{1}{t} \E\big[ (F_t - F)^4 -(F_t - F)^2 \big] \xrightarrow{t\downarrow 0}
2 \sum_{k=1}^m  p_kq_k \E[ z_k^2(z_k^2-1)],
\end{align}
which leads to a different formula for $\cJ(F)$ 
from \eqref{J_sp}.
 See Remark \ref{rem_useJ} for an application.

\end{remark}

 Now we are ready to state the main bound of the paper.

\begin{theorem}[\textsf{Master bounds}]
\label{thm_master}
Fix $r\in\N$ and
let $F\in L^4(\PP)$ belong to the $r$-th chaos in the Poisson setting
or Rademacher setting.
Recall $\cD(F)$ and $\fE_r(F)$ from \eqref{def_CD} and \eqref{E_abs}.
Assume that a nonnegative quantity $\rho(F)$ satisfies

\noi
\begin{align}
 \|J_{2r}(F^2)\|_2^2\geq2m_2(F)^2-\rho(F).
 \label{top_rho}
\end{align}

\noi
Put

\noi
\begin{align}
 \cT_F:=\frac1r\Gamma(F,F)-m_2(F)-\frac F2
 \AND
 \mathfrak d(F):=\abs{\cD(F)}+\rho(F).
 \label{frak_d}
\end{align}
Then,
\begin{align}
 \fE_r(F)&\leq\mathfrak{d}(F),
 \label{E_master}\\
 \abs{m_3(F)-m_2(F)}
 &\leq\sqrt{m_2(F)\mathfrak d(F)},
 \label{m3_master}\\
 \norm{\cT_F}_2
 &\leq\frac{2r-1}{2r}\sqrt{\mathfrak d(F)}.
 \label{T_master}
\end{align}
Moreover, without any lattice assumption,
\begin{align}
 \cJ(F)
 &\leq4\sqrt{m_2(F)\mathfrak d(F)}
   +12\mathfrak d(F).
 \label{J_master}
\end{align}
If, in addition, $\rho(F)=0$ and $\cD(F)\geq0$, then
$12\mathfrak d(F)$ in \eqref{J_master} can be replaced by
$8\mathfrak d(F)$.

If $F$ additionally satisfies the lattice condition \eqref{cond_La}
with shift $\theta>0$, then, with
$c_\theta:=(1-e^{-\theta})/\theta$,
\noi
\begin{align}
 \dTV(F+\theta,P_\theta)
 &\leq c_\theta\bigg[
 \abs{m_2(F)-\theta}
 +\frac{2r-1}{2r}\sqrt{\mathfrak d(F)}
 +\frac23\sqrt{m_2(F)\mathfrak d(F)}
 +2\mathfrak d(F)
 \bigg].
 \label{master_TV}
\end{align}
In the preceding special case $\rho(F)=0$ and $\cD(F)\geq0$,
$2\mathfrak d(F)$ in \eqref{master_TV} can be replaced by
$\frac43\mathfrak d(F)$.
\end{theorem}

\begin{proof}
Put
\[
 \mathcal Q(F):=\norm{J_{2r}(F^2)}_2^2-2m_2(F)^2
 \geq-\rho(F).
\]
By \eqref{D_ide}--\eqref{E_abs},
$\fE_r(F)=\cD(F)-\mathcal Q(F)\geq0$.  Since
$-\mathcal Q(F)\leq\rho(F)$,
\[
 \fE_r(F)
 \leq\abs{\cD(F)}+\rho(F)
 =\mathfrak d(F),
\]
which proves \eqref{E_master}.  Moreover,
\[
 m_3(F)-m_2(F)
 =\E\big[(J_r(F^2)-F)F\big].
\]
Cauchy--Schwarz, \eqref{E_abs}, and \eqref{E_master} give
\eqref{m3_master}.

From the chaos decomposition \eqref{rel_Gamma},
\noi
\begin{align*}
 \norm{\cT_F}_2^2
 &=\sum_{k=1}^{2r-1}
 \left(1-\frac{k}{2r}\right)^2
 \norm{J_k(F^2)}_2^2\ind_{\{k\ne r\}}
 +\frac14\norm{J_r(F^2)-F}_2^2\\
 &\leq\left(\frac{2r-1}{2r}\right)^2\fE_r(F),
\end{align*}
which proves \eqref{T_master}.

We next derive \eqref{J_master}; no lattice condition is used here.
Set
\[
 \mathcal N_r(F)
 :=6\norm{J_r(F^2)-F}_2^2
 +\frac6r\sum_{\substack{1\leq k\leq2r-1\\k\ne r}}
 k\norm{J_k(F^2)}_2^2\geq0.
\]
A direct rearrangement of \eqref{J_sp}, using \eqref{def_CD} and
\eqref{sp_2}, gives
\[
 \cJ(F)
 =8\cD(F)-12\mathcal Q(F)
 +4\big[m_3(F)-m_2(F)\big]-\mathcal N_r(F).
\]
Since $-\mathcal Q(F)\leq\rho(F)$,
\[
 8\cD(F)-12\mathcal Q(F)
 \leq8\abs{\cD(F)}+12\rho(F)
 \leq12\mathfrak d(F).
\]
Together with \eqref{m3_master}, this proves \eqref{J_master}.
If $\rho(F)=0$ and $\cD(F)\geq0$, then
$\mathcal Q(F)\geq0$, and the same identity yields the sharper bound
\noi
\begin{align}
 \cJ(F)
 \leq4\sqrt{m_2(F)\cD(F)}+8\cD(F).
\end{align}

Assume now that $F$ satisfies \eqref{cond_La}.  Let $F_t$ be the
exchangeable pair from Proposition~\ref{prop_J}, put
$D_t:=F_t-F\in\Z$, and apply Lemma~\ref{lem_Stein} to
$(F+\theta,F_t+\theta)$ with $a=1-e^{-rt}$ and $R=0$.  This gives
\noi
\begin{align*}
 \dTV(\cL(F+\theta),P_\theta)
 \leq c_\theta\bigg\{
 \E\abs{
 \frac{\E[D_t^2\mid F]}{2(1-e^{-rt})}
 -\frac{F+2\theta}{2}}
 +\frac{\E[D_t^2(D_t^2-1)]}{6(1-e^{-rt})}
 \bigg\}.
\end{align*}
By Proposition~\ref{prop_J}, the second term converges to
$\cJ(F)/6$.  Propositions~\ref{prop_P_pair} and \ref{prop_R_pair},
together with the definition of $\cT_F$, give
\noi
\begin{align}
 \limsup_{t\downarrow0}
 \E\abs{
 \frac{\E[D_t^2\mid F]}{2(1-e^{-rt})}
 -\frac{F+2\theta}{2}}
 \leq\abs{m_2(F)-\theta}+\norm{\cT_F}_2.
 \label{cond_master}
\end{align}
Combining \eqref{T_master}, \eqref{J_master}, and its sharper version
with the preceding two displays proves \eqref{master_TV} and the final
assertion.
\qedhere
\end{proof}

We now verify \eqref{top_rho} in the two settings.

\begin{lemma}
\label{lem_P_top}
Let $F=I_q^{\eta}(f)\in L^4(\PP)$.
Then,

\noi
\begin{align}
 \| J_{2q}(F^2) \|_2^2    \geq2m_2(F)^2.
\notag
\end{align}
\end{lemma}

\begin{proof} 
The highest projection is $I_{2q}^{\eta}(f\sot f)$,
with $f\sot f$ the   canonical symmetrization of $f\otimes f$.
It is known that 

\noi
\begin{align}
 \| J_{2q}(F^2) \|_2^2 = (2q)! \|  f\sot f \|^2_{L^2(\mu^{2q})}
 &=  (q!)^2\sum_{s=0}^q
 \binom qs^2\norm{f\ot_s f}^2_{L^2(\mu^{2q-2s})} \notag \\
 &= 2m_2(F)^2 +  (q!)^2\sum_{s=1}^{q-1}
 \binom qs^2\norm{f\ot_s f}^2_{L^2(\mu^{2q-2s})},
 \label{P_c_s}
\end{align}
where $f \ot_s f$ denotes the   $s$-contraction between $f$ and itself; 
see, e.g., \cite[Lemma 5.1]{DVZ18}.
Hence, the proof is completed. 
\qedhere

\end{proof}

\begin{lemma}
\label{lem_R_top}

Let $F$ belong to the $d$-th Rademacher chaos and put
$\cM(F)=\sup_k\Inf_k(F)$ as in \eqref{Rad_inf}.  Then
\noi
\begin{align}
 A_{2d} :=  \| J_{2d}( F^2) \|_2^2
 \geq2m_2(F)^2
 - (2d)!\binom{2d}{d}m_2(F) \cM(F).
 \label{R_top}
\end{align}
\end{lemma}

\begin{proof}
Assume $F=Q_d(h;Y)$ for some $h\in\fH_0^{\odot d}$ as in \eqref{Rad_chaos}.
Let $\Delta_p = \{ (i_1, \ldots, i_p)\in\N^p: i_k \neq i_j$ for different $k, j\}$ as in \cite[(1.3)]{Zhe19}.
Then, we can first write 

\noi
\begin{align*}
A_{2d} &= (2d)! \| h \sot h \ind_{\Delta_{2d}} \|^2_{\fH^{\otimes 2d}}
 =  (2d)! \| h \sot h \|^2_{\fH^{\otimes 2d}}
- (2d)! \| h \sot h \ind_{\Delta_{2d}^c} \|^2_{\fH^{\otimes 2d}} \\
&\geq 2 m_2(F)^2 -  (2d)! \| h \sot h \ind_{\Delta_{2d}^c} \|^2_{\fH^{\otimes 2d}}, 
\end{align*}
where the last inequality follows from the same argument as in Lemma \ref{lem_P_top}.

For the second term, we deduce from  \cite[(3.3)]{Zhe19}
that 
\begin{align*}
  \| h \sot h \ind_{\Delta_{2d}^c} \|^2_{\fH^{\otimes 2d}}
&\leq \sum_{r=1}^d r! \binom{d}{r}^2 \|h\|^2_{\fH^{\otimes d}} \cM(h) 
\leq  \binom{2d}{d} m_2(F) \cM(F).
\end{align*}
Hence, the desired bound \eqref{R_top} follows immediately. 
\qedhere

\end{proof}

The following two  consequences of Theorem~\ref{thm_master} 
are immediate.

\begin{corollary}[\textsf{Four-moment bound on Poisson chaos}]
\label{cor_P_quant}
Fix $q\in\N$ and 
let $F\in L^4(\PP)$ be in the $q$-th Poisson chaos
such that $F$ satisfies
 the lattice condition \eqref{cond_La}
with shift $\theta>0$.  
Then,
$\cP(F)\geq0$ and with $c_\theta:= (1-e^{-\theta})/\theta$,

\noi
\begin{align}
 &\dTV\bigl( F+\theta, P_\theta\bigr)
\leq c_\theta\bigg[
 |m_2(F)-\theta|
 +\frac{2q-1}{2q}\sqrt{\cP(F)}
 +\frac23\sqrt{m_2(F)\cP(F)}
 +\frac43\cP(F)
 \bigg].
\notag
\end{align}
\end{corollary}

\begin{proof}
Note that  Lemma~\ref{lem_P_top} and
the identity \eqref{D_ide} imply $\cP(F)\geq0$.
Then, the proof is completed by 
invoking  Theorem~\ref{thm_master} with
$\rho(F)=0$ and $\cD(F)=\cP(F)$.    
\qedhere

\end{proof}

\begin{corollary}[\textsf{Four-moment-influence bound on Rademacher chaos}]
\label{cor_R_quant}
Let $F\in L^4(\PP)$ belong to the $d$-th Rademacher chaos
and satisfy  the lattice condition \eqref{cond_La}
with shift $\theta>0$.  
Put
\noi
\begin{align}
 \mathfrak d(F)
 :=|\cR(F)|
 + (2d)!\binom{2d}{d}m_2(F)\cM(F)
\notag
\end{align}
with $\cM(F)$ the maximal influence as in \eqref{Rad_inf}.
Then, with   $c_\theta:= (1-e^{-\theta})/\theta$,

\noi
\begin{align}
 &\dTV\bigl( F+\theta,P_\theta\bigr)
  \leq c_\theta\bigg[
 |m_2(F)-\theta|
 +\frac{2d-1}{2d}\sqrt{\mathfrak{d}(F)}
 +\frac23\sqrt{m_2(F)\mathfrak{d}(F)}
 +2\mathfrak{d}(F)
 \bigg].
\notag\end{align}
\end{corollary}

\begin{proof}
We can apply 
  Lemma~\ref{lem_R_top} and Theorem~\ref{thm_master} with
$ \rho(F)= (2d)!\binom{2d}{d}m_2(F)\cM(F)$
and $\cD(F)=\cR(F)$.
\qedhere

\end{proof}

Finally, we present the proofs of 
Theorem~\ref{thm_P}
and 
Theorem~\ref{thm_R}.

\begin{proof}[Proof of Theorem~\ref{thm_P}]
Let us first prove the equivalence between the assertions (i) and (ii). 
By Lemma~\ref{lem_raw}, assertion \textup{(i)} implies  (ii).  
Conversely, \eqref{D_ide} and
Lemma~\ref{lem_P_top} give
\noi
\begin{align}
 \fE_q(F_n)\leq\cP(F_n).
\notag
\end{align}
Thus \eqref{m3_master} yields $m_3(F_n)\to\lambda$, and solving
\eqref{def_P} for $m_4(F_n)$ yields
$m_4(F_n)\to3\lambda^2+\lambda$. 
Therefore, (i) follows from   Lemma~\ref{lem_raw}.

Now we have proved the    equivalence between the assertions (i) and (ii). 
Assume either (i) or (ii). 
Then,
Corollary~\ref{cor_P_quant} and \eqref{Poi_param} imply
the  total-variation convergence 
$F_n + \theta_n \to P_\lambda$,
which leads to the convergence in law
\[
(F_n + \theta_n)^4\xrightarrow{\rm law} P_\lambda^4
\]
by continuous mapping theorem. 
Since we have $\E[(F_n+\theta_n)^4]\to\E[P_\lambda^4]$,
 Lemma~\ref{lem_UI} therefore yields
uniform integrability of $\{ |F_n+\theta_n|^4\}_{n\geq 1}$.
That is, assertion (iii) is proved. 
Conversely, total-variation convergence and
uniform integrability of the fourth powers 
imply convergence of all
moments of order at most four. This proves (i),
and hence completes the proof. 
\qedhere

\end{proof}

\begin{proof}[Proof of Theorem~\ref{thm_R}]
Lemma \ref{lem_raw} shows that  (i) implies (ii).
  Under \eqref{M_zero}, Lemma~\ref{lem_R_top} and
\eqref{D_ide} give

\noi
\begin{align}
 \mathfrak{d}(F_n)
 \leq|\cR(F_n)|
 + (2d)!\binom{2d}{d}m_2(F_n) \cM(F_n)
 \to0.
\notag\end{align}
The bounds \eqref{m3_master} and \eqref{def_R} then give the 
third-moment and fourth-moment limits, so \textup{(ii)} implies
\textup{(i)}.  Corollary~\ref{cor_R_quant}, followed by the same
uniform-integrability argument as in the proof of Theorem~\ref{thm_P}, 
proves the remaining equivalences.
\end{proof}

\subsection{Lower-order perturbations in the two settings}
\label{SEC_33}

For a centered random variable
$H=\sum_{j=1}^rH_j$, with $H_j$ in the $j$-th chaos of either the
Poisson or Rademacher space, we define

\noi
\begin{align}
\begin{aligned}
 \cJ_r(H)
 :&=\frac1r\big(6\E[H^2L(H^2)]+2\E[HLH]-8\E[H^3LH]\big).
\end{aligned}
 \label{Jq_def}
\end{align}
For a pure $r$-th chaos variable, we have
$\cJ_r(F)=\cJ(F)$ whenever the latter is defined.

The following estimates make the perturbation argument quantitative.

\begin{lemma}
\label{lem_pert_est}
Fix $\mathsf{S}\in\{\mathsf{P},\mathsf{R}\}$.  Let
$U\in\cC_j^{\mathsf{S}}\cap L^4(\PP)$ and
$V\in\cC_k^{\mathsf{S}}\cap L^4(\PP)$.  Then,
\noi
\begin{align}
 \norm{\Gamma(U,V)}_2
 \leq\frac{j+k}{2}\norm{U}_4\norm{V}_4.
 \label{Gamma_L4}
\end{align}

Let $H=\sum_{j=1}^rH_j$, with
$H_j\in\cC_j^{\mathsf{S}}\cap L^4(\PP)$, and let $H_t$ be obtained
from the thinning coupling when $\mathsf{S}=\mathsf{P}$ and from the
coordinate-refresh coupling when $\mathsf{S}=\mathsf{R}$.  Without
any lattice assumption,
\noi
\begin{align}
 \cJ_r(H)
 =\lim_{t\downarrow0}
 \frac{\E[(H_t-H)^2\{(H_t-H)^2-1\}]}
 {1-e^{-rt}}.
 \label{Jq_limit}
\end{align}
If $H$ satisfies the lattice condition \eqref{cond_La} with shift
$\theta>0$, then $\cJ_r(H)\geq0$.

If $F\in\cC_r^{\mathsf{S}}\cap L^4(\PP)$ and
$B=\sum_{j=1}^{r-1}B_j$, with
$B_j\in\cC_j^{\mathsf{S}}\cap L^4(\PP)$, then
\noi
\begin{align}
 \abs{\cJ_r(F+B)-\cJ_r(F)}
 \leq84\bigl(1+\norm{F}_4+\eps\bigr)^3\eps,
 \label{Jq_cont}
\end{align}
where
$\eps:=\sum_{1\leq j<r}\norm{B_j}_4$.

Assume that a nonnegative quantity $\rho(F)$ satisfies
\eqref{top_rho}, and put
$\mathfrak d(F):=\abs{\cD(F)}+\rho(F)$ as in
\eqref{frak_d}.  Then, without any lattice assumption,
\noi
\begin{align}
 \cJ_r(F)
 \leq4\sqrt{m_2(F)\mathfrak d(F)}+12\mathfrak d(F)
 \label{Jr_pure_up}
\end{align}
and
\begin{align}
 |\cJ_r(F)|
 \leq
 4\sqrt{m_2(F)\mathfrak d(F)}
 +16\mathfrak d(F).
 \label{Jq_pure}
\end{align}

In the Poisson setting, one has the sharper estimates
\noi
\begin{align}
 \cJ_r(F)
 &\leq4\sqrt{m_2(F)\cP(F)}+8\cP(F).
 \label{Jq_pure_up}
\end{align}
\end{lemma}

\begin{proof}
The product $UV$ belongs to the sum of the first $j+k$ chaoses in
both settings.  Since
\noi
\begin{align*}
 \Gamma(U,V)
 =\frac12\big[L(UV)+(j+k)UV\big]
 =\frac12\sum_{\ell=0}^{j+k}(j+k-\ell)J_\ell(UV),
\end{align*}
orthogonality and H\"older's inequality give \eqref{Gamma_L4}.

Let $P_t$ be the corresponding Ornstein--Uhlenbeck semigroup.  As in
\eqref{m2Dt}--\eqref{m4Dt},
\noi
\begin{align*}
 \E[(H_t-H)^2]
 &=2\big(\E[H^2]-\E[HP_tH]\big),\\
 \E[(H_t-H)^4]
 &=2\E[H^4]-8\E[H^3P_tH]+6\E[H^2P_t(H^2)].
\end{align*}
Differentiating at zero and using $1-e^{-rt}\sim rt$ proves
\eqref{Jq_limit}.  If $H+\theta$ is integer-valued, then so is
$H_t+\theta$, and hence
\[
 (H_t-H)^2\big[(H_t-H)^2-1\big]\geq0.
\]
This proves the asserted nonnegativity.

For the continuity estimate, put $H=F+B$.  Since
$\norm{B}_4\leq\eps$ and $\norm{LB}_4\leq r\eps$, one has
\noi
\begin{align}\label{est11}
\begin{aligned}
 \norm{H^3-F^3}_{4/3}
 &\leq3(\norm{F}_4+\eps)^2\eps,
 \\
 \norm{L(H^2)-L(F^2)}_2
 &\leq2r(2\norm{F}_4+\eps)\eps,
 \\
 \norm{L(H^2)}_2
 &\leq2r(\norm{F}_4+\eps)^2.
\end{aligned}
\end{align}
Here we used that $H^2$ and $F^2$ belong to the sum of the first
$2r$ chaoses.  From \eqref{Jq_def},
\noi
\begin{align*}
 r\big(\cJ_r(H)-\cJ_r(F)\big)
 &={}
 6\E\big[H^2L(H^2)-F^2L(F^2)\big]\\
 &\quad+2\E\big[HLH-FLF\big]
 -8\E\big[H^3LH-F^3LF\big].
\end{align*}
The estimates in \eqref{est11} yield
\noi
\begin{align*}
 \abs{\E\big[H^2L(H^2)-F^2L(F^2)\big]}
 &\leq8r(\norm{F}_4+\eps)^3\eps,\\
 \abs{\E[HLH-FLF]}
 &\leq r(2\norm{F}_4+\eps)\eps,\\
 \abs{\E[H^3LH-F^3LF]}
 &\leq4r(\norm{F}_4+\eps)^3\eps.
\end{align*}
Combining them gives \eqref{Jq_cont}.

For a pure $r$-th chaos, $\cJ_r(F)=\cJ(F)$.  Therefore
\eqref{Jr_pure_up} is precisely \eqref{J_master}, whose proof is
purely spectral and uses no lattice assumption.  In the Poisson
setting, $\rho(F)=0$ and $\cD(F)=\cP(F)$, so the sharper part of
Theorem~\ref{thm_master} gives \eqref{Jq_pure_up}.

It remains to prove the absolute estimate \eqref{Jq_pure}.  Put
\[
 \mathcal Q(F)
 :=\norm{J_{2r}(F^2)}_2^2-2m_2(F)^2
\]
and define $\mathcal N_r(F)$ as in the proof of
Theorem~\ref{thm_master}.  By \eqref{D_ide},
\[
 \fE_r(F)=\cD(F)-\mathcal Q(F).
\]
Hence the same exact identity can be rewritten as
\noi
\begin{align}
 \cJ_r(F)
 &=-4\cD(F)+12\fE_r(F)
 +4\big[m_3(F)-m_2(F)\big]-\mathcal N_r(F).
\end{align}
By the definition of $\mathcal N_r(F)$ and \eqref{E_abs},
\noi
\begin{align}
 0\leq\mathcal N_r(F)\leq12\fE_r(F).
\end{align}
Therefore,
\noi
\begin{align*}
 \abs{\cJ_r(F)}
 &\leq
 4\abs{\cD(F)}
 +12\fE_r(F)
 +4\abs{m_3(F)-m_2(F)}.
\end{align*}
Using \eqref{E_master}, \eqref{m3_master}, and
$\abs{\cD(F)}\leq\mathfrak d(F)$, we obtain
\noi
\begin{align*}
 \abs{\cJ_r(F)}
 &\leq
 4\sqrt{m_2(F)\mathfrak d(F)}
 +16\mathfrak d(F),
\end{align*}
which proves \eqref{Jq_pure}.
\qedhere

\end{proof}

\begin{proof}[Proof of Theorem~\ref{thm_pert}]
Fix $n$ with $\theta_n>0$, suppress the index $n$, and put
$H:=F+B$, $X:=\theta+H$, $D_t:=H_t-H$, and
$a_t:=1-e^{-rt}$, where $H_t$ is obtained from the common coupling in
the relevant space.  Let $\mathcal G_0$ be the sigma-field generated
by the underlying Poisson process or Rademacher sequence.

The exact regressions on the individual chaoses give
\noi
\begin{align}
 \E[D_t | \mathcal{G}_0]
 =-a_tH+R_t,
 \quad
 R_t:=\sum_{j=1}^{r-1}
 \big[a_t-(1-e^{-jt})\big]B_j.
\end{align}
Since $H=X-\theta$ is $\sigma\{X\}$-measurable, the tower property
yields
\noi
\begin{align}
 \E[D_t | X]
 =-a_t(X-\theta)+\widetilde{R}_t,
 \qquad
 \widetilde{R}_t:=\E[R_t  | X].
 \label{pert_wtR}
\end{align}
For $j<r$,
$0\leq a_t-(1-e^{-jt})=e^{-jt}-e^{-rt}\leq a_t$; hence
conditional Jensen's inequality gives
\noi
\begin{align}
 \frac{\E | \widetilde R_t|}{a_t}
 \leq\frac{\E\abs{R_t}}{a_t}
 \leq\sum_{j=1}^{r-1}\norm{B_j}_1
 \leq\eps,
 \label{pert_R_bdd}
\end{align}
where $\eps:=\sum_{j=1}^{r-1}\norm{B_j}_4$.

By \eqref{P_cross_reg} or \eqref{R_cross_reg},
\noi
\begin{align}
 A_t
 :=\frac{\E[D_t^2 | \mathcal{G}_0]}{2a_t}
 \xrightarrow{t\downarrow0}\frac1r\Gamma(H,H)
 \qquad\text{in }L^2(\PP).
\end{align}
Taking conditional expectations with respect to $X$ and using Jensen's
inequality, we obtain
\noi
\begin{align}
 \limsup_{t\downarrow0}
 \E\abs{
 \frac{\E[D_t^2\mid X]}{2a_t}-\frac{X+\theta}{2}}
 \leq
 \norm{\frac1r\Gamma(H,H)-\frac H2-\theta}_1.
 \label{pert_cond}
\end{align}
Expanding
$\Gamma(H,H)=\Gamma(F,F)+2\Gamma(F,B)+\Gamma(B,B)$ and applying
\eqref{Gamma_L4} term by term gives
\noi
\begin{align}
 \norm{\frac1r\Gamma(H,H)-\frac H2-\theta}_1
 &\leq
 \abs{m_2(F)-\theta}
 +\norm{\cT_F}_2
 +2\norm{F}_4\eps+\eps^2+\frac\eps2.
 \label{pert_cond_bdd}
\end{align}

Since $X=H+\theta\in\N_0$ and $H_t+\theta$ has the same law as
$H+\theta$, one has $D_t\in\Z$.  Lemma~\ref{lem_pert_est} therefore
yields
\noi
\begin{align}
 \lim_{t\downarrow0}
 \frac{\E[D_t^2(D_t^2-1)]}{a_t}
 =\cJ_r(H)\geq0.
 \label{pert_jump}
\end{align}
Applying Lemma~\ref{lem_Stein} to
$(X,X_t)=(H+\theta,H_t+\theta)$, using
\eqref{pert_R_bdd}, \eqref{pert_cond}, \eqref{pert_cond_bdd}, and
\eqref{pert_jump}, and then letting $t\downarrow0$, gives

\noi
\begin{align}
 \dTV(X,P_\theta)
 &\leq
 c_\theta\bigg[
 \abs{m_2(F)-\theta}
 +\norm{\cT_F}_2
 +2\norm{F}_4\eps+\eps^2+\frac\eps2
 +\frac{\cJ_r(F+B)}6
 \bigg]
 \notag\\
 &\qquad\qquad+
 \left(1+\frac{c_\theta}{2}\right)\eps.
 \label{pert_quant}
\end{align}

In the Poisson setting take $\rho(F)=0$; in the Rademacher setting take
\[
 \rho(F)=(2r)!\binom{2r}{r}m_2(F)\cM(F).
\]
Theorem~\ref{thm_master}, Lemmas~\ref{lem_P_top}--\ref{lem_R_top},
and Lemma~\ref{lem_pert_est} then give, with
$\mathfrak d=\mathfrak d_{\mathsf S}(F)$,
\[
 \norm{\cT_F}_2
 \leq\frac{2r-1}{2r}\sqrt{\mathfrak d}
\]
and
\noi
\begin{align}
 \frac{\cJ_r(F+B)}6
 &\leq
 \frac23\sqrt{m_2(F)\mathfrak d}
 +b_{\mathsf S}\mathfrak d
 +14\bigl(1+\norm{F}_4+\eps\bigr)^3\eps.
\end{align}
Substitution into \eqref{pert_quant} proves \eqref{pert_main_bdd} and
\eqref{pert_rem}.

It remains to prove the asymptotic assertions.  Exactly as in the
pure-chaos criteria, \eqref{m3_master} and \eqref{def_CD} yield
\noi
\begin{align}
 \abs{m_3(F_n)-m_2(F_n)}
 \leq\sqrt{m_2(F_n)\mathfrak d_{\mathsf S}(F_n)}
 \longrightarrow0,
\end{align}
and hence
\noi
\begin{align}
 m_4(F_n)
 &=3m_2(F_n)^2+2m_3(F_n)-m_2(F_n)+\cD(F_n)
\longrightarrow3\lambda^2+\lambda.
\end{align}
Thus, for some $n_0$,
\noi
\begin{align}
 \sup_{n\geq n_0}\norm{F_n}_4<\infty.
\end{align}
Since $\theta_n\to\lambda>0$ and $\eps_n\to0$, the remainder in
\eqref{pert_rem} tends to zero, and \eqref{pert_main_bdd} gives
\noi
\begin{align}
 \dTV(X_n,P_{\theta_n})\longrightarrow0.
\end{align}
Together with \eqref{Poi_param}, this implies
\noi
\begin{align}
 \dTV(X_n,P_\lambda)\longrightarrow0.
 \label{pert_TV_lambda}
\end{align}

Finally, $\norm{B_n}_4\leq\eps_n\to0$, so the first four moments of
$X_n=\theta_n+F_n+B_n$ have the same limits as those of
$\theta_n+F_n$.  Lemma~\ref{lem_raw} and the preceding centered-moment
limits give
\noi
\begin{align}
 \E[X_n^k]\longrightarrow\E[P_\lambda^k],
 \qquad1\leq k\leq4.
 \label{pert_mom}
\end{align}
Now \eqref{pert_TV_lambda} and \eqref{pert_mom}, followed by
Lemma~\ref{lem_UI} with $Y_n=X_n^4$, prove uniform integrability.
\end{proof}

\subsection{Higher-order and elementary applications}
\label{SEC_34}

\begin{proof}[Proof of Corollary~\ref{cor_hyper}]
Write
$
 Z_{n,i,j}:=N_{n,i,j}-\alpha_{n,i,j}.
 $
Expanding the product in \eqref{hyper_count} gives
\[
 X_n
 =\theta_n+F_n+\sum_{s=1}^{q-1}B_{n,s}
 \]
 with
\begin{align*}
 \theta_n
 :=\sum_{i=1}^{m_n}\beta_{n,i},
\quad
 F_n
 :=\sum_{i=1}^{m_n}\prod_{j=1}^qZ_{n,i,j},
\AND
 B_{n,s}
 :=\sum_{i=1}^{m_n}
 \sum_{\substack{S\subseteq\{1,\ldots,q\}\\|S|=s}}
 \left(\prod_{j\notin S}\alpha_{n,i,j}\right)
 \prod_{j\in S}Z_{n,i,j}.
\end{align*}
The independent Poisson variables $N_{n, i,j}$'s    may be
realized on disjoint sets of a Poisson random measure.  Hence $F_n$
belongs to the $q$-th chaos and $B_{n,s}$ to the $s$-th chaos.

For
$V_{n,i}:=\prod_{j=1}^qZ_{n,i,j}$, independence within a block and
the centered Poisson moments \eqref{center_m} give

\noi
\begin{align*}
 m_2(V_{n,i}) = m_3(V_{n,i})&=\beta_{n,i}
\AND
 m_4(V_{n,i}) =\beta_{n,i}  \prod_{j=1}^q(1+3\alpha_{n,i,j}).
\end{align*}

\noi
Since cumulants are additive over the independent blocks,
we have $ m_2(F_n) = \cum_2(F_n) =\sum_i\beta_{n,i}$
and

\noi
\begin{align*}
 \cP(F_n) & = \cum_4(F_n) - 2\cum_3(F_n) + \cum_2(F_n) \\
 &=\sum_{i=1}^{m_n}
 \left[ \beta_{n,i}
 \bigg[ - 1 + \prod_{j=1}^q(1+3\alpha_{n,i,j}) \bigg]
             -3\beta_{n,i}^2  \right].
\end{align*}
For all sufficiently large $n$, 
\begin{align}\label{large_n}
\delta_n\leq1  
\AND 
\sum_{i=1}^{m_n} \beta_{n,i} \leq 2\lambda.
\end{align}
Since
$\cP(F_n)\geq0$ by Lemma~\ref{lem_P_top}, and
\[
 \prod_{j=1}^q(1+3\alpha_{n,i,j})
 \leq (1 + 3\delta_n)^q 
 \leq 1 + (4^q-1)\delta_n,
\]

\noi
we obtain

\noi
\begin{align}  \label{hyper_P_bdd}
 0\leq\cP(F_n)
 \leq (4^q-1)\delta_n\sum_{i=1}^{m_n} \beta_{n,i}
 \leq  2(4^q-1) \lambda\delta_n
 \to0.
\end{align}

It remains to control the lower chaoses.  For $|S|=s$, put

\noi
\begin{align}
 U_{n,i,S}
 :=\left(\prod_{j\notin S}\alpha_{n,i,j}\right)
   \prod_{j\in S}Z_{n,i,j} 
\AND 
 W_{n,i,s}:=\sum_{|S|=s}U_{n,i,S}.
\notag
\end{align}
It is easy to see that
different subsets are orthogonal
with
$
\E [ U_{n,i, S}^2  ] = \beta_{n,i} \prod_{j\notin S} \alpha_{n,i,j}
$,
and 
thus,  
\begin{align}\label{bdd_w2}
 \E[W_{n,i,s}^2]
 \leq 2^q\beta_{n,i}\delta_n^{q-s}.
 \end{align}
 
 \noi
Moreover, when $\delta_n\leq1$,
$
\E[Z_{n,i,j}^4]
 =\alpha_{n,i,j}+3\alpha_{n,i,j}^2
 \leq4\alpha_{n,i,j}.
$
Using
$(\sum_{\ell=1}^M|x_\ell|)^4
 \leq M^3\sum_{\ell=1}^M|x_\ell|^4$, we get

\noi
\begin{align}\label{bdd_w4}
 \E[W_{n,i,s}^4]
& \leq \binom{q}{s}^3 \sum_{|S|=s} \E[U_{n,i, S}^4]
\leq 2^{3q} 4^s\sum_{|S|=s}  \bigg(\prod_{j\notin S} \alpha_{n,i,j}^4\bigg)\prod_{j\in S} \alpha_{n, i, j}
  \notag   \\
 &\leq 2^{6q}\beta_{n,i}\delta_n^{3(q-s)}.
\end{align}

\noi
The random variables $W_{n,i,s}$ are centered and independent over $i$.
Consequently, we deduce from \eqref{bdd_w2}
and \eqref{bdd_w4} that  with \eqref{large_n} for sufficiently large $n$,

\noi
\begin{align}
 \E[B_{n,s}^4] 
 &= \sum_{i=1}^{m_n}  \E[W_{n,i,s}^4]
 +  3\sum_{i\neq j}  \E[W_{n,i,s}^2] \E[W_{n,j,s}^2]  
  \notag \\
 &\leq 2^{6q}\bigg( \sum_{i=1}^{m_n}\beta_{n,i} \bigg) \delta_n^{3(q-s)} 
  + 3  \Bigg(  2^q \sum_{i=1}^{m_n}\beta_{n,i}\delta_n^{q-s} \Bigg)^2 \notag  \\
 &\leq   \big( 2^{6q+1}\lambda + 2^{2q+4}\lambda^2  \big) \delta_n^2
  \to0.
 \label{hyper_lower_bdd}
\end{align}
All hypotheses of Theorem~\ref{thm_pert} are satisfied.

We finally record the rate.  Since
$\theta_n=\sum_i\beta_{n,i}\to\lambda$, the quantities
$\theta_n$, $c_{\theta_n}$ and $\norm{F_n}_4$ are uniformly bounded
for all sufficiently large $n$, with $\theta_n$ also bounded away
from zero.  By \eqref{hyper_P_bdd}, the pure-chaos part on the
right-hand side of \eqref{pert_main_bdd} is bounded by
$C_{q,\lambda}\sqrt{\delta_n}$.  Taking fourth roots in
\eqref{hyper_lower_bdd} and summing over $s$ yields
$\eps_n\leq C_{q,\lambda}\sqrt{\delta_n}$, where the constant
$C_{q,\lambda}$ only depends on $q$ and $\lambda$ and may vary from
line to line in this proof.  Therefore,
$\mathfrak R_{\theta_n}(F_n,\eps_n)
\leq C_{q,\lambda}\sqrt{\delta_n}$, and
\eqref{pert_main_bdd} implies
$
 \dTV\bigl(X_n, P_{\theta_n} \bigr)
 \leq C_{q,\lambda}\sqrt{\delta_n}.
$
Combining this with \eqref{Poi_param}, we can  obtain
\eqref{hyper_rate}.
\qedhere

\end{proof}

\begin{proof}[Proof of Corollary~\ref{cor_small}]
Put
$F_n=X_n-\theta_n =\sum_{k=1}^{N_n}(B_{n,k}-p_{n,k})$,
which belongs to the first Rademacher chaos. 
By direct computations, we have

\noi
\begin{align*}
 m_2(F_n)-\theta_n
 =-\sum_kp_{n,k}^2\to0
 \AND
 \cM(F_n)=\max_kp_{n,k}(1-p_{n,k})\to0.
\end{align*}

\noi
  For a centered Bernoulli
variable $B-p$, the cumulants of orders two, three, and four 
are
\[
 p(1-p),
 \qquad p(1-p)(1-2p),
 \qquad p(1-p)\big[ 1-6p(1-p)\big].
\]
Cumulant additivity under independent sums gives

\noi
\begin{align*}
 \cR(F_n) 
 &= \cum_4(F_n)  - 2\cum_3(F_n) + \cum_2(F_n)  \\
 &=2\sum_kp_{n,k}^2(1-p_{n,k})(3p_{n,k}-1),
\end{align*}
which  is bounded by
$8(\max_kp_{n,k})\theta_n$.  
Then, we can conclude the proof by Theorem~\ref{thm_R}.
\qedhere

\end{proof}

\begin{proof}[Proof of Corollary~\ref{cor_blocks}]
For a centered random variable $F$, the moment defect 
$
 \cD(F)=\cum_4(F)-2\cum_3(F)+\cum_2(F)
$
is additive under independent sums. 
As a result, 
$
 \cD(F_n)=\sum_k\cD(F_{n,k}).
 $
Then, 
the Poisson assertion follows from Theorem~\ref{thm_P} with $\cD = \cP$.  In the
Rademacher case, disjoint coordinate blocks imply that the maximal
influence of the sum is the maximum of the block-coordinate
influences; then 
the Rademacher assertion follows from Theorem~\ref{thm_R} with $\cD = \cR$.
\qedhere

\end{proof}

\section{Further examples and proofs}
\label{SEC4}

\subsection{An exact quadratic counterexample}
\label{SEC_41}

\begin{proposition}[\textsf{Exact four-moment counterexample}]
\label{prop_counter}
Let $\eps_1, \eps_2, \eps_3$, $\eps_a$, $\eps_b$, $\eps_c$, and
$\eps_d$   be i.i.d. symmetric
Rademacher random variables. 
Define
\[
 Q_\triangle
 :=\eps_1\eps_2
 +\eps_1\eps_3
 +\eps_2\eps_3
 \AND
  Q_\square
 :=\eps_a\eps_b
 +\eps_b\eps_c
 +\eps_c\eps_d
 +\eps_d\eps_a.
\]
Take four independent copies of $Q_\triangle$ and three independent
copies of $Q_\square$, and set their independent sums to be
\noi
\begin{align}
 F:=\sum_{j=1}^4Q_{\triangle,j}
   +\sum_{j=1}^3Q_{\square,j}.
 \notag
\end{align}
Then, $F$ belongs to the second Rademacher chaos
and satisfies

\noi
\begin{align}
 \frac{F+24}{2}\in\N_0
 \quad\text{and}\quad
 \E[(F+24)^k]=\E[P_{24}^k],
 \qquad 1\leq k\leq4,
  \notag
\end{align}
where $P_{24}\sim\Pois(24)$.  Moreover,

\noi
\begin{align}
 \dTV\bigl(\cL(F+24),\Pois(24)\bigr)
 \geq\frac{1-e^{-48}}2.
 \notag
\end{align}
\end{proposition}

\begin{proof}
Note that $Q_\triangle$  is  a triangle edge sum,
whose
 distribution is given by 
 $
 \PP(Q_\triangle=3)=\frac14
 $
and
$\PP(Q_\triangle=-1)=\frac34$.
It is easy to compute that
$\cum_2(Q_\triangle)=3$,
$\cum_3(Q_\triangle)=6$,
$ \cum_4(Q_\triangle)=-6.$
Similarly, $ Q_\square$ is a  four-cycle edge sum
with 
$ \PP(Q_\square=4)
 =\PP(Q_\square=-4)=\frac18$
 and
$
 \PP(Q_\square=0)=\frac34$,
and hence
$ \cum_2(Q_\square)=4,$ 
$\cum_3(Q_\square)=0$,
$\cum_4(Q_\square)=16.$
Additivity of cumulants under independent sums
 gives
$
 \cum_2(F)=\cum_3(F)=\cum_4(F)=24.
$
Therefore,
 $m_2(F)=m_3(F)=24$,
$m_4(F)=3\cdot24^2+24$,
and
$\cR(F)=0.$
The  identities \eqref{raw_id} show that
$F+24$ has the first four moments of $\Pois(24)$.

On the other hand, each triangle block is at least $-1$ and each
four-cycle block is at least $-4$, so $F+24 \geq8$.  Moreover, $F$ is a
sum of twenty-four signs and is therefore even-valued.  Thus $F+24$ is
even-valued and
\[
 \dTV\bigl(\cL(F+24),\Pois(24)\bigr)
 \geq\PP\big(\Pois(24)\text{ is odd} \big)
=\frac{1-e^{-48}}{2},
\]
where the last equality 
follows from 
$\E[ (-1)^{P_\lambda} ] = e^{-2\lambda}$ and 
$\E[ (-1)^{P_\lambda} ] = \PP\big( P_\lambda\text{ is even} \big)
 -  \PP\big( P_\lambda\text{ is odd} \big)$
 with $ \PP\big( P_\lambda\text{ is even} \big)
 +  \PP\big( P_\lambda\text{ is odd} \big) =1$.
Hence, the proof   is completed. 
\qedhere

\end{proof}

\begin{remark}[\textsf{The unit-jump obstruction in Proposition \ref{prop_counter}}] 
\label{rem_useJ}
In the construction of $F$, there are four triangles and three four-cycles. 
Note that every vertex of their disjoint union 
has degree two.  Flipping one sign
changes $F$ by $0$ or $\pm4$.  
Let $k$ denote one of the  twenty-four vertices.
Then, using the notation $z_k$ from \eqref{zk_def}, we get
\noi
\begin{align*}
 z_k^2=4(\eps_u+\eps_v)^2,
\end{align*}
where $u,v$ are the two neighbors of $k$ in a triangle or four-cycle.
Thus,
$\PP(z_k^2=16)=\PP(z_k=0)=1/2$.
Formula \eqref{J_R_coord} gives

\noi
\begin{align}
 \cJ(F) =\frac{1}{4} \sum_k \E[ z_k^2 (z_k^2 -1) ] 
 = \frac{1}{4} \times 24 \times [4^2 \times (4^2-1) ] \times\frac{1}{2}
= 720.
 \notag
\end{align}
Thus the four moments do not detect the large coordinate jumps, while
vanishing maximal influence does.
\end{remark}

\subsection{Two moments and total variation are not enough}
\label{SEC_42}

\begin{example}[\textsf{Two moments do not suffice}]
Let $B_{n,1},\ldots,B_{n,n}$ be independent Bernoulli variables with
$p_{n,k}=1/(2n)$ and put
\noi
\begin{align}
 X_n:=1+2\sum_{k=1}^nB_{n,k},
 \qquad
 F_n:=X_n-2.
 \notag
\end{align}
Then $F_n$ belongs to the first   Rademacher chaos,
$\E[X_n]=2$, and
\noi
\begin{align}
 \Var(X_n)=2\left(1-\frac1{2n}\right)\to2,
 \qquad
 \cM(F_n)\to0.
 \notag
\end{align}
Nevertheless, by the law of small numbers,
\noi
\begin{align}
 X_n\xrightarrow{\mathrm{law}}1+2\Pois(1/2),
 \notag
\end{align}
which is supported on the odd integers and is not $\Pois(2)$.
\end{example}

\begin{example}[\textsf{The role of uniform integrability}]

Fix $\lambda>0$.  Let
$S_n\sim\text{Binomial}(n,\lambda/n)$ and
$B_n\sim\operatorname{Bernoulli}(n^{-4})$ be independent, and set
$
 X_n:=S_n+nB_n
 $
 and
 $
 F_n:=X_n-\E[X_n].
$
Then $F_n$ belongs to the first Rademacher chaos,
$\E[X_n]=\lambda+n^{-3}\to\lambda$, and
\noi
\begin{align}
 \cM(F_n)
 =\max\left\{
 \frac\lambda n\left(1-\frac\lambda n\right),
 n^{-2}(1-n^{-4})
 \right\}
 \to0.
 \notag
\end{align}
Moreover,
\noi
\begin{align}
 \dTV\bigl(\cL(X_n),\Pois(\lambda)\bigr)
 \leq
 \dTV\bigl(\cL(S_n),\Pois(\lambda)\bigr)+n^{-4}
 \to0.
 \notag
\end{align}
However, a direct calculation yields
$
 \E[X_n^4]
 \to
 \lambda^4+6\lambda^3+7\lambda^2+\lambda+1.
$
Thus total-variation convergence, even together with vanishing maximal
influence, does not imply fourth-moment convergence without uniform
integrability.
\end{example}

\subsection{Exact versus limiting Poisson laws}
\label{SEC_43}

The lattice condition in the main four-moment criteria is an exact
structural requirement.  We first identify the pure-chaos variables
that realize an exact Poisson law, and then turn to limiting centered
Poisson behavior when the exact lattice structure is absent.

\begin{example}[\textsf{Exact Poisson laws in the first Poisson chaos}]
\label{ex_first_P}

Let $\eta$ be a Poisson random measure on $(\cZ, \sZ)$ with control measure $\mu$
 and let   $F=I_1^\eta(f) \in L^4(\PP)$ for some $f\in L^2(\mu)$.
 Then, necessarily we have 
$f\in  L^4(\mu)$ as well.\footnote{For $G= I_q(g)\in\cC^{ \mathsf{P}}_q$
with $g\in L^2_s(\mu^q)$, the fourth integrability of $G$ forces 
$g\in L^4(\mu^q)$; see, e.g., \cite[Theorem 1.8]{Zhe26}.}
The standard moment identities, as a consequence 
of the product formula \cite{DP18b},
 give us

\noi
\begin{align}
 m_2(F)=\int_\cZ f^2 d\mu,
 \quad
 m_3(F)=\int_\cZ f^3 d\mu,
\AND
 m_4(F)
 =3\left(\int_\cZ f^2 d\mu\right)^2
   +\int_\cZ f^4 d\mu.
\notag
\end{align}
Therefore
\noi
\begin{align}
 \cP(F)
 =\int_\cZ f(z)^2\bigl(f(z)-1\bigr)^2\,\mu(dz).
 \label{P_zero_one}
\end{align}
In particular, $\cP(F)=0$ if and only if
$f=\ind_A$ almost everywhere for some measurable set $A$ with
$\mu(A)<\infty$.  In that case

\noi
\begin{align}
 F=\eta(A)-\mu(A)
\AND
 F+\mu(A)=\eta(A)\sim\Pois(\mu(A)).
\notag\end{align}
Formula \eqref{P_zero_one} gives a concrete interpretation of the
``one-unit'' property.  Indeed,  
with $D^+$ the add-one cost operator,
we have 

\noi
\begin{align*}
 D^+_zF
 =
 F(\eta+\delta_z)-F(\eta)
 =
 f(z),
\end{align*}
so that $f(z)$ is precisely the jump of the functional produced by
adding one Poisson atom at $z$.  Hence $\cP(F)=0$ forces
$f(z)\in\{0,1\}$ for $\mu$-almost every $z\in\cZ$: an atom is either ignored
or counted exactly once.  Any other weight produces a non-unit jump
and a strictly positive contribution to the defect,
and thus prevents a Poisson law.
\end{example}

\begin{remark}[\textsf{Rigidity beyond the first Poisson chaos}]
\label{rem_exact_high_P}
The preceding phenomenon is specific to the first Poisson chaos.
Indeed, let
\[
 F=I_q^\eta(f)\in\cC_q^{ \mathsf{P}}\cap L^4(\PP),
 \qquad q\geq2.
\]
We first deduce from \eqref{D_ide} and \eqref{P_c_s} in Lemma \ref{lem_P_top}
that 

\noi
\begin{align*}
 \cP(F) 
 \geq
 (q!)^2
 \sum_{r=1}^{q-1}
 \binom{q}{r}^2
 \|f\otimes_r f\|_2^2.
\end{align*}
Hence $\cP(F)=0$ implies
$f\otimes_{1}f=0$, which forces $f=0$ and thus
$F=0$ almost surely.\footnote{Indeed, the implication
$f\otimes_1f=0 \Rightarrow  f=0$
follows directly from a Hilbert--Schmidt operator argument.
Define
$
 T_f:L^2(\mu)\longrightarrow L^2(\mu^{q-1})
$
by
$ (T_fh)(x_1,\ldots,x_{q-1}):=\int_{\cZ}  f(x_1,\ldots,x_{q-1},z)h(z)\,\mu(dz).$
Then $T_fT_f^*$ is the integral operator whose kernel is
$f\otimes_1f$.  
Hence $f\otimes_1f=0$ implies $T_fT_f^*=0$.  
Consequently, for every $g\in L^2(\mu^{q-1})$,
$ \|T_f^*g\|_2^2 =  \langle T_fT_f^*g,g\rangle =0.$
Thus $T_f^*=0$ and hence $T_f=0$.  
Since $\|T_f\|_{\rm HS}^2=\|f\|_{L^2(\mu^q)}^2$, 
it follows that $f=0$ almost everywhere.}  
As a result,  a nonzero Poisson chaos of order
$q\geq2$ always satisfies
\noi
$
 \cP(F)>0.
$
Thus exact Poisson laws occur in a pure Poisson chaos only at first
order.
\end{remark}

The strict positivity in Remark~\ref{rem_exact_high_P} is not uniform
on the unit-variance sphere of a higher Poisson chaos.  The next two
examples give, for every fixed $q\geq2$, nondegenerate variables in the
$q$-th Poisson chaos whose Poisson moment defect tends to zero.  In fact,
both sequences converge in law to a centered Poisson random variable.

\begin{example}[\textsf{A rare-cell Charlier construction}]
\label{ex_rare_charlier}
Fix an integer $q\geq2$.  Let $\eta$ be a Poisson random measure on
$\R_+$ with Lebesgue control measure.  Put
\[
 \lambda_n:=\frac1n,
 \qquad
 M_n:=q!\,n^q=q!\lambda_n^{-q},
 \qquad
 A_{n,j}:=\left(\frac{j-1}{n},\frac jn\right],
 \quad 1\leq j\leq M_n.
\]
Define
\[
 f_n
 :=\frac1{q!}\sum_{j=1}^{M_n}\ind_{A_{n,j}}^{\ot q},
 \qquad
 F_n:=I_q^\eta(f_n),
\]
and write
\[
 Z_{n,j}
 :=\frac1{q!}I_q^\eta\bigl(\ind_{A_{n,j}}^{\ot q}\bigr),
 \qquad
 F_n=\sum_{j=1}^{M_n}Z_{n,j}.
\]
The cells $A_{n,j}$ are disjoint with the same length, so the centered random variables
$Z_{n,1},\ldots,Z_{n,M_n}$ are independent with the same law.  Since
$\cP=\cum_4-2\cum_3+\cum_2$ and cumulants are additive over
independent sums,
$
 \cP(F_n)=M_n\cP(Z_{n,1}).
$

We first estimate the moment defect of one cell directly from the Poisson
product formula \cite{DP18b},
to which we refer readers for any unexplained notation.
  Let $A$ be a measurable set with
$\mu(A)=\lambda\leq1$ and put
\[
 g_\lambda:=\frac1{q!}\ind_A^{\ot q},
 \qquad
 Z:=I_q^\eta(g_\lambda).
\]
By the Poisson isometry,
we have 
$
 m_2(Z)
 =q!\norm{g_\lambda}_2^2
 =\frac{\lambda^q}{q!}
 =:v_\lambda.
$
For $0\leq\ell\leq r\leq q$, we have
the following simple expression for the
$(r,\ell)$-star contraction
\[
 g_\lambda\star_r^\ell g_\lambda
 =
 \frac{\lambda^\ell}{(q!)^2}
 \ind_A^{\ot(2q-r-\ell)}  
\]
and hence
$
 \big\| g_\lambda\star_r^\ell g_\lambda \big\|_2^2
 =
 \frac{\lambda^{2q-r+\ell}}{(q!)^4}.
$
In particular,
$
 q!\,g_\lambda\star_q^0g_\lambda=g_\lambda.
$
By the product formula for multiple Poisson integrals
(see, e.g., 
\cite[Proposition 2.1]{DP18b}), we can write 
\[
 Z^2
 =
 \sum_{r=0}^q
 r!\binom qr^2
 \sum_{\ell=0}^r
 \binom r\ell
 I_{2q-r-\ell}^\eta
 \left( g_\lambda \star_r^\ell g_\lambda \right),
\]
where $g_\lambda \star_r^\ell g_\lambda$ is symmetric.
The term corresponding to $(r,\ell)=(q,0)$ is exactly $Z$.
Thus we may write
\[
 Z^2=Z+R_\lambda,
\]
where $R_\lambda$ is the sum of all the remaining terms.
For every $(r,\ell)\neq(q,0)$,
\[
 2q-r+\ell\geq q+1.
\]
By the Poisson isometry and the contraction estimate above, 
each remaining
summand in the product formula has squared $L^2(\PP)$-norm bounded by
\[
 C_q\lambda^{2q-r+\ell}
 \leq C_q\lambda^{q+1}.
\]
Since only finitely many summands occur, with their number depending
only on $q$, the Cauchy--Schwarz inequality gives
$
 \E[R_\lambda^2]
 =
 \E\bigl[(Z^2-Z)^2\bigr]
 =
 O_q(\lambda^{q+1}).
$
Since $Z$ is centered, we get
\[
 \cP(Z) =
 \E\bigl[(Z^2-Z)^2\bigr]-3m_2(Z)^2.
\]
As $q\geq2$,
$
 m_2(Z)^2
 =
 O_q(\lambda^{2q})
 =
 O_q(\lambda^{q+1}),
$
and therefore
$
 |\cP(Z)|
 =
 O_q(\lambda^{q+1}).
$
Consequently,
\[
 \cP(F_n)
 =
 M_n\cP(Z_{n,1})
 =
 O_q(\lambda_n)
 =
 O_q(n^{-1}).
\]
Moreover, since $F_n$ is a nonzero element of the $q$-th Poisson
chaos, Remark~\ref{rem_exact_high_P} shows that
\[
 0<\cP(F_n)=O_q(n^{-1})\longrightarrow0.
\]
The normalization is exact:
$
 m_2(F_n)
 =
 M_n\frac{\lambda_n^q}{q!}
 =
 1$
 for every $n$.

The product formula also gives the corresponding third- and
fourth-moment asymptotics.  Indeed, in the computation of
$\E[Z^3]$, only the $q$-th chaos projection of $Z^2$ contributes.
The relevant pairs satisfy $r+\ell=q$.  The pair
$(r,\ell)=(q,0)$ contributes $v_\lambda$, whereas every other such
pair has $\ell\geq1$ and contributes
$O_q(\lambda^{q+1})$.  Hence
\[
 m_3(Z)
 =
 v_\lambda+O_q(\lambda^{q+1}).
\]
Since $Z^2=Z+R_\lambda$, we further have
\begin{align*}
 m_4(Z)
 &=\E\bigl[(Z+R_\lambda)^2\bigr] 
 =m_2(Z)
   +2\bigl[m_3(Z)-m_2(Z)\bigr]
   +\E[R_\lambda^2] \\
 &=v_\lambda+O_q(\lambda^{q+1}).
\end{align*}
It follows from the additivity of cumulants that
\[
 \cum_2(F_n)=1,
 \qquad
 \cum_3(F_n)=1+O_q(n^{-1}),
 \qquad
 \cum_4(F_n)=1+O_q(n^{-1}),
\]
or, equivalently,
$m_2(F_n)=1,$ 
$m_3(F_n)=1+O_q(n^{-1}),$
and
$ m_4(F_n)=4+O_q(n^{-1}).
$

We finally identify the limiting law.  Put
\[
 K_n
 :=\sum_{j=1}^{M_n}
 \ind_{\{\eta(A_{n,j})=q\}}
\AND
 \vartheta_n:=\E[K_n].
\]
If
$
 p_n
 :=e^{-\lambda_n}\frac{\lambda_n^q}{q!},
$
then
\[
 K_n\sim\operatorname{Binomial}(M_n,p_n)
\AND
 \vartheta_n
 =M_np_n
 =e^{-\lambda_n}
 \longrightarrow1.
\]
Hence the law of small numbers gives
$
 K_n\xrightarrow{\rm law}
 \Pois(1).
$

It remains to compare $F_n$ with the centered rare-cell count.
Let $N\sim\Pois(\lambda)$ and let $Z$ be as above.  The Charlier
representation gives\footnote{The Charlier formula can be obtained
by the relation $\sum_{q\geq 0} C_q(x,\lambda) \frac{t^q}{q!} = e^{-\lambda t} (1+t)^x$;
see, e.g., \cite[Proposition 6.2.8]{Pri09}. }
\[
 Z
 =
 \frac1{q!}C_q(N,\lambda)
 =
 \frac1{q!}
 \sum_{s=0}^q
 \binom qs(-\lambda)^{q-s}(N)_s,
\]
where $(N)_0:=1$ and
$
 (N)_s:=N(N-1)\cdots(N-s+1),$
for $s\in\N$.
In particular,
$
 Z(q)=1+O_q(\lambda).
$
If
\[
 p:=e^{-\lambda}\frac{\lambda^q}{q!},
 \qquad
 Y:=\ind_{\{N=q\}}-p,
\]
then
\[
 \begin{split}
 \E[(Z-Y)^2]
 &=
 \E[Z^2]+p(1-p)-2pZ(q)
 =
 O_q(\lambda^{q+1}),
 \end{split}
\]
where we used
$
 \E[Z^2]=\frac{\lambda^q}{q!},$
and
 $p
 =\frac{\lambda^q}{q!}
 +O_q(\lambda^{q+1}),
$
together with $Z(q)=1+O_q(\lambda)$ and
$p^2=O_q(\lambda^{2q})$.
Applying this estimate independently on the $M_n$ disjoint cells
gives
\noi
\begin{align}
 \E\bigl[
 \{F_n-(K_n-\vartheta_n)\}^2
 \bigr]
 =
 M_nO_q(\lambda_n^{q+1})
 =
 O_q(n^{-1}).
 \label{rare_L2}
\end{align}
Since $\vartheta_n\to1$, the law of small numbers and the preceding
$L^2$ estimate yield $1 + F_n\xrightarrow{\rm law} \Pois(1)$.
Thus, for every fixed $q\geq2$, we have constructed a unit-variance
sequence in the $q$-th Poisson chaos such that
$
 \cP(F_n)=O_q(n^{-1}) \to 0
$
and
$1 + F_n\xrightarrow{\rm law} \Pois(1)$.

\end{example}

The sparse $q$-clique count considered below is a standard
Poisson $U$-statistic, and its Poisson asymptotics follow from
general results on sparse subgraph counts in random geometric
graphs; see, e.g., \cite{BP14,DST16}.  Our purpose here is
different: we consider its highest-chaos projection and show
directly that it provides a nondegenerate sequence in the pure
$q$-th Poisson chaos whose Poisson moment defect tends to zero.

\begin{example}[\textsf{Sparse geometric clique projections}]
\label{ex_geo_clique}
Fix $q\geq2$ and $d\geq1$.  Let $\eta$ be a stationary Poisson point
process of unit intensity on $\R^d$.  With $y_0:=0$, define
\noi
\begin{align}
 a_{d,q}
 :=\int_{(\R^d)^{q-1}}
 \ind_{\left\{
 \max_{0\leq i<j\leq q-1}\norm{y_i-y_j}\leq1
 \right\}}
 \dd y_1\cdots\dd y_{q-1}.
\notag 
\end{align}
The constant $a_{d,q}$ is finite and strictly positive: the domain of
integration is contained in $B(0,1)^{q-1}$ and contains
$B(0,1/2)^{q-1}$.

Put
\noi
\begin{align}
 L_n:=n^{1/d},
 \quad
W_n:=[0,L_n]^d,
 \AND
 r_n:=
 \left(\frac{q!}{a_{d,q}n}\right)^{1/[d(q-1)]}.
\notag 
\end{align}
Define the symmetric kernel

\noi
\begin{align}
 h_n(x_1,\ldots,x_q)
 :=\ind_{\{x_1,\ldots,x_q\in W_n\}}
 \ind_{\left\{
 \max_{1\leq i<j\leq q}\norm{x_i-x_j}\leq r_n
 \right\}}.
\notag 
\end{align}
Let $\eta_{\ne}^q$ denote the set of ordered $q$-tuples of distinct
points of $\eta$.  Then

\noi
\begin{align}
 U_n
 :=\frac1{q!}
 \sum_{(x_1,\ldots,x_q)\in\eta_{\ne}^q}
 h_n(x_1,\ldots,x_q)
\notag 
\end{align}
is the number of $q$-cliques in the random geometric graph on
$\eta\cap W_n$ with connection radius $r_n$.
Set

\noi
\begin{align}
 F_n
 :=I_q^\eta\left(\frac{h_n}{q!}\right)
\AND
 \vartheta_n
 :=\E[U_n]
 =\frac1{q!}\int_{(\R^d)^q}h_n.
\notag 
\end{align}
By the Wiener--It\^o expansion of a Poisson $U$-statistic
(see, e.g., \cite[Proposition~12.11]{LP18}),
\[
 U_n-\vartheta_n
 =\sum_{j=1}^q I_j^\eta(g_{n,j}),
 \qquad
 g_{n,q}=\frac{h_n}{q!}.
\]
Hence $F_n$ is the highest-chaos projection of
$U_n-\vartheta_n$.

For all sufficiently large $n$, $2r_n<L_n$.  If the first point is at
least distance $r_n$ from the boundary of $W_n$, anchoring it and
rescaling the remaining $q-1$ points gives the lower bound below; the
same change of variables without the boundary restriction gives the
upper bound:

\noi
\begin{align}
 a_{d,q}r_n^{d(q-1)}(L_n-2r_n)^d
 \leq\int h_n
 \leq a_{d,q}r_n^{d(q-1)}L_n^d.
\notag 
\end{align}
Since
$
 a_{d,q}n r_n^{d(q-1)}=q!,$
we obtain
$\big(1-\frac{2r_n}{L_n}\big)^d
 \leq\vartheta_n\leq1$,
 so that 
 $\vartheta_n\to 1$.
Because $h_n^2=h_n$, the Poisson isometry gives
\noi
\begin{align}
 m_2(F_n)
 =q! \big\| \frac{h_n}{q!} \big\|_{L^2((\R^d)^q)}^2
 =\frac1{q!}\int h_n
 =\vartheta_n.
 \notag 
\end{align}

We next compute the third and fourth cumulants.  Fix
$m\in\{3,4\}$.  We use the partition notation of
\cite[Section~3]{LPST14}.  Arrange the $mq$ arguments into the
$m$ rows
\[
 J_i:=\{(i-1)q+1,\ldots,iq\},
 \qquad 1\leq i\leq m.
\]
Let $\Pi(q,\ldots,q)$ be the set of partitions $\sigma$ of
$[mq]$ such that
\[
 \abs{B\cap J_i}\leq1
 \qquad
 \text{for every block $B\in\sigma$ and every $i$}.
\]
Let $\widetilde\Pi_{\geq2}(q,\ldots,q)$ denote the subset of
$\Pi_{\geq2}(q,\ldots,q)$ consisting of those partitions $\sigma$
for which the following graph on $\{1,\ldots,m\}$ is connected:
two row indices $i$ and $j$ are joined whenever some block of
$\sigma$ contains one element of $J_i$ and one element of $J_j$.
 For a function $f$ on
$(\R^d)^{mq}$, write $f_\sigma$ for the function obtained by
identifying arguments whose indices belong to the same block of
$\sigma$.\footnote{For example, if $m=3$, $q=2$, and
$
 \sigma=\big\{\{1,3,5\},\{2,4,6\}\big\},
$
then
$
 (h_n^{\otimes3})_\sigma(y_1,y_2)
 =h_n(y_1,y_2)^3
 =h_n(y_1,y_2).$
 }

Since $h_n$ is bounded and has bounded support,
\cite[Theorem~3.1]{LPST14} applies to
$F_n=I_q^\eta(h_n/q!)$ and gives
\[
 \cum_m(F_n)
 =
 \frac1{(q!)^m}
 \sum_{\sigma\in\widetilde\Pi_{\geq2}(q,\ldots,q)}
 \int (h_n^{\otimes m})_\sigma\,
 d\mu^{\abs{\sigma}}.
\]
Put $b:=\abs{\sigma}$.  Since no block contains two positions from
the same row, every block has size at most $m$; as there are $mq$
positions altogether, necessarily $b\geq q$.

If $b=q$, every block meets every row exactly once.  After labeling
the blocks by their positions in the first row, each of the remaining
$m-1$ rows can be matched to them in $q!$ ways.  Hence there are
$(q!)^{m-1}$ such partitions.  For each of them, symmetry and
$h_n^m=h_n$ imply that the associated integral equals $\int h_n$.
Their total contribution is therefore
\[
 (q!)^{m-1}(q!)^{-m}\int h_n
 =\vartheta_n.
\]

Suppose that $b\geq q+1$.  
Form a graph whose vertices are the blocks
of $\sigma$, joining two blocks whenever they occur in the same row.
Connectedness of $\sigma$ implies that this graph is connected, and
each edge forces the corresponding spatial variables to be at
distance at most $r_n$.  Choosing a spanning tree and integrating
successively from its root gives
\[
 \int (h_n^{\otimes m})_\sigma\,d\mu^{b}
 \leq
 C_{d,q,m}\,n r_n^{d(b-1)}.
\]
Here the root contributes at most $\mu(W_n)=n$, while each of the
$b-1$ tree edges contributes at most the volume of a ball of radius
$r_n$.  Since
\[
 a_{d,q}n r_n^{d(q-1)}=q!,
\]
for $b\geq q+1$ we have
\[
 n r_n^{d(b-1)}
 =
 \frac{q!}{a_{d,q}}r_n^{d(b-q)}
 =
 O_{d,q}(r_n^d).
\]
There are only finitely many admissible partitions for fixed $q$ and
$m$.  Consequently,
\[
 \cum_3(F_n)
 =\vartheta_n+O_{d,q}(r_n^d)
 \AND
 \cum_4(F_n)
 =\vartheta_n+O_{d,q}(r_n^d).
\]
Together with
$\cum_2(F_n)=m_2(F_n)=\vartheta_n$, this yields
\[
 0<\cP(F_n)
 =
 O_{d,q}(r_n^d)
 =
 O_{d,q}\bigl(n^{-1/(q-1)}\bigr)
 \longrightarrow0.
\]
The strict inequality follows from
Remark~\ref{rem_exact_high_P}, since $F_n$ is a nonzero element of the
$q$-th Poisson chaos.  As $\vartheta_n\to1$, we also obtain
$
 \big( m_2(F_n),
 m_3(F_n),
 m_4(F_n) \big)
 \longrightarrow (1,1,4).
$

Finally, the same partition argument applies to every fixed
$m\geq2$ and gives
$ \cum_m(F_n)
 =
 \vartheta_n+O_{d,q,m}(r_n^d).
$
Thus,
$
 \cum_1(F_n)=0$
and
$\cum_m(F_n)\longrightarrow1$,
for every $m\geq2$.
Since the centered Poisson law is moment-determinate, 
we have
$1 + F_n \to \Pois(1)$
in law. 

\end{example}

\begin{remark}[\textsf{Scope of the preceding constructions}]
\label{rem_high_P_scope}
In both examples, $F_n$ is an exact element of the fixed $q$-th
Poisson chaos.  In Example~\ref{ex_rare_charlier}, it is
$L^2(\PP)$-asymptotically equivalent to $K_n-\vartheta_n$ by
\eqref{rare_L2}; in Example~\ref{ex_geo_clique}, it is the
highest-chaos projection of $U_n-\vartheta_n$.  However,
$F_n+\vartheta_n$ is not, in general, nonnegative integer-valued.
Consequently, these examples do not directly satisfy the lattice
condition \eqref{cond_La}, and Theorem~\ref{thm_P} cannot be invoked to
obtain total-variation convergence for $F_n+\vartheta_n$.  Their
purpose is instead to show that the strict inequality $\cP(F)>0$ for
nonzero higher-order Poisson chaoses is not uniform: for every fixed
$q\geq2$, one can have
\noi
\begin{align}
 m_2(F_n)\longrightarrow1,
 \qquad
 \cP(F_n)\longrightarrow0,
 \qquad
 F_n\xrightarrow{\rm law}P_1-1.
 \notag 
\end{align}
\end{remark}

The two constructions show that centered Poisson laws can occur as
weak limits of pure higher-order Poisson chaoses even though exact
representation is impossible.  We next return to exact representation
on the Rademacher space.  The law of small numbers shows that centered
Poisson laws are limits of first-chaos Rademacher variables.  Exact
representation is nevertheless impossible.

\begin{proposition}[\textsf{No exact Poisson law in the first Rademacher chaos}]
\label{prop_no_exact_R}

Let $F$ be a random variable in the first Rademacher chaos. 
Then, for every $\lambda>0$,
$
 \cL(F+\lambda)\ne\Pois(\lambda).
$

\end{proposition}

\begin{proof}
If $F=0$, or if $F$ depends on only finitely many coordinates, then
$F+\lambda$ takes only finitely many values and hence cannot have a
Poisson distribution.  We may therefore assume that $F$ depends on
infinitely many coordinates.  Writing $F$ as a linear statistic of
independent Bernoulli random variables, let $(B_k)_{k\geq1}$ be
independent Bernoulli variables with
$p_k:=\PP(B_k=1)\in(0,1)$, and set
$
 Y_k:=\frac{B_k-p_k}{\sqrt{p_k(1-p_k)}}.
$
Then
\[
 F=\sum_{k\geq1}c_kY_k
\quad\text{with}
\quad
 \sum_{k\geq1}c_k^2<\infty.
\]
After discarding the zero coefficients and relabeling the remaining
coordinates, we may further assume that
$
 c_k\neq0
 $
for every  $k\in\N.
$

Argue by contradiction and assume 
$X:=F+\lambda\sim\Pois(\lambda)$ for some $\lambda >0$.
With $c_k\neq 0$, we set

\noi
\begin{align}
 a_k:=\frac{c_k}{\sqrt{p_k(1-p_k)}},
\notag
\end{align}
and then we write
$
 X=Z_k+a_kB_k,
$
where $Z_k$ is independent of $B_k$.  Since $X$ is $\N_0$-valued and
$B_k = 0$ with positive probability, Fubini's theorem
shows that $Z_k\in\N_0$ almost surely, and thus $a_k\in\Z\setminus\{0\}$
with $Z_k+a_k\in\N_0$.

If $a_k>0$, put $m_k=a_k$, $W_k=Z_k$ and
$\widetilde B_k=B_k$.  If $a_k<0$, put
$m_k=-a_k$, $W_k=Z_k+a_k$ and
$\widetilde B_k=1-B_k$.  In both cases,

\noi
\begin{align}
 X=W_k+m_k\widetilde B_k,
\label{exact_R_pos}
\end{align}
where $m_k\in\N$, $W_k\in\N_0$, and $W_k$ is independent of the
nondegenerate Bernoulli variable $\widetilde B_k$.  Moreover,
$0\leq W_k\leq X$ almost surely.

Let $G_k(z):=\E[z^{W_k}]$ for $z\in\mathbb{C}$.
For every $R\geq1$, the bound $W_k\leq X$ gives
$ \E[R^{W_k}]
 \leq\E[R^X]
 =\exp[\lambda(R-1)]$.
This implies that $G_k$ is an entire function on the complex plane. 
Since
$r_k:=\PP(\widetilde B_k=1)\in(0,1)$, independence in
\eqref{exact_R_pos} yields, first on the unit disk and then everywhere
by analytic continuation,
$
 \exp\{\lambda(z-1)\}
 =G_k(z)\bigl(1-r_k+r_kz^{m_k}\bigr).
$
The polynomial on the right has a complex zero, whereas the
exponential on the left never vanishes,
which leads to the desired     contradiction.
\qedhere

\end{proof}

\subsection{Bounded-bias coordinate rigidity}
\label{SEC_44}

In this section, we present the proof of Theorem \ref{thm_bias}.
For this purpose, we need the following.

\begin{lemma}[\textsf{Zero-set bound}]
\label{lem_zero}
Let
$
 P\in\bigoplus_{j=0}^s\cC_j^{\mathsf{R}}
$
be nonzero with $s\in\N_0$, and let $r:=\deg(P)\leq s$.
Then, for every
$(i_1,\ldots,i_r)\in\Delta_r$ such that
$
 D_{i_r}\cdots D_{i_1}P\not\equiv 0,
$
one has

\noi
\begin{align}
 \PP(P\ne0)
 \geq
 \prod_{\ell=1}^r
 \min\{p_{i_\ell},q_{i_\ell}\}.
 \label{zero_local}
\end{align}
Consequently, if, for some
$\eps\in(0,\frac12]$,
\[
 \eps\leq p_k\leq1-\eps
\]
for every coordinate on which $P$ depends, then
\noi
\begin{align}
 \PP(P\ne0)
 \geq\eps^r
 \geq\eps^s.
 \label{zero_bdd}
\end{align}
\end{lemma}

\begin{proof}
We prove the local estimate \eqref{zero_local} by induction on
$r=\deg(P)$.  If $r=0$, then $P$ is a nonzero constant and hence
$\PP(P\ne0)=1$.

Suppose $r\geq1$, and fix
$(i_1,\ldots,i_r)\in\Delta_r$ such that
$
 D_{i_r}\cdots D_{i_1}P\not\equiv0.
$
Put $k=i_1$.  Write

\noi
\begin{align*}
 P=A+Y_kB,
\end{align*}
where $A$ and $B$ do not depend on $Y_k$.  With our normalization of
the discrete gradient, $B=D_kP$.  In particular,

\noi
\begin{align*}
 D_{i_r}\cdots D_{i_2}B
 =
 D_{i_r}\cdots D_{i_1}P
 \not\equiv0.
\end{align*}

Conditionally on all coordinates except $Y_k$, the two possible
values of $P$ cannot both vanish on $\{B\ne0\}$.  Therefore

\noi
\begin{align*}
 \PP(P\ne0 | (Y_j)_{j\ne k})
 \geq
 \min\{p_k,q_k\}\ind_{\{B\ne0\}}.
\end{align*}
Taking expectations gives

\noi
\begin{align*}
 \PP(P\ne0)
 \geq
 \min\{p_k,q_k\}\PP(B\ne0).
\end{align*}
Applying the induction hypothesis to $B=D_kP$ and the coordinates
$i_2,\ldots,i_r$, we obtain

\noi
\begin{align*}
 \PP(P\ne0)
 &\geq
 \prod_{\ell=1}^r
 \min\{p_{i_\ell},q_{i_\ell}\}.
\end{align*}
This proves \eqref{zero_local}.  The uniform bound
\eqref{zero_bdd} follows immediately.
\end{proof}

\begin{proof}[Proof of Theorem \ref{thm_bias}]
Let us recall the notations $D_k, z_k, \Inf_k$ from 
\eqref{Dk_def}, \eqref{zk_def}, \eqref{Dk_rel},
and \eqref{D_inf}.

Fix a coordinate $k$ with $\Inf_k(F)>0$.  
By the representation
\eqref{Rad_chaos}, the discrete gradient $D_kF$ is a nonzero
homogeneous Rademacher polynomial of degree $d-1$.  Applying
Lemma~\ref{lem_zero} gives
\[
 \PP(D_kF\ne0)
 \geq\eps^{d-1}.
\]
Since $F+\theta$ is integer-valued,
$
 z_k=(p_kq_k)^{-1/2}D_kF
$
is integer-valued.  Therefore

\noi
\begin{align}
d\,d! \Inf_k(F)
 &=\E[(D_kF)^2]
 =p_kq_k\E[z_k^2]
 \geq p_kq_k\PP(z_k\ne0)
 \notag\\
 &\geq
 \eps(1-\eps)\eps^{d-1}
 =
 \eps^d(1-\eps),
\notag 
\end{align}
where we used the concavity of the map $p\in[\eps, 1-\eps]\mapsto p(1-p)$.
Therefore,  \eqref{inf_lower} is proved. 

In view of the second identity in \eqref{D_inf},
the number of   coordinates with positive influence
 is bounded by $d\, m_2(F) \eps^{-d}(1-\eps)^{-1}$.

For a sequence $F_n\in\cC^{\mathsf{R}}_d$ with uniformly bounded variances, 
choose an integer
\[
K >  \frac{d  }{\eps^d  (1-\eps)} \sup_n m_2(F_n).
\]   Every shifted chaotic component
 then has at most $2^K$ support
points.  If a subsequence converged weakly to
$P_\lambda\sim\Pois(\lambda)$ with $\lambda>0$, choose the
$2^K+1$ pairwise disjoint intervals

\noi
\begin{align*}
 (j-\tfrac13,j+\tfrac13),
 \qquad
 0\leq j\leq2^K.
\end{align*}
Their boundaries have zero $P_\lambda$-mass, while each interval has
positive $P_\lambda$-mass.  Weak convergence would therefore force
the support of every sufficiently large member of the subsequence to
meet all $2^K+1$ intervals, contradicting the fact that it has at most
$2^K$ points.

Hence, the proof is completed. 
\end{proof}

\begin{remark}[\textsf{Bias rigidity versus unit-jump rigidity}]
The rare-event mechanism behind Poisson approximation takes rather
different forms on Rademacher and Poisson space.

\smallskip
(i) In the Rademacher setting, consider  a sequence 
$F_n\in\cC^{\mathsf{R}}_d$ with symmetric kernels $f_n$
such that $m_2(F_n)$ is bounded away from zero and 
$\cM(F_n)\to0$.
Then, Theorem~\ref{thm_bias} shows that the
lattice condition \eqref{cond_La}, together with
$\cM(F_n)\to0$, forces every nonzero term of the homogeneous sum to contain at least
one coordinate whose Bernoulli parameter approaches either zero or
one.
 More precisely, whenever
\[
 f_n(i_1,\ldots,i_d)\ne0,
\]
the local zero-set estimate \eqref{zero_local} implies that at least one of the
coordinates $i_1,\ldots,i_d$ must become increasingly biased.  In
particular,
\noi
\begin{align}
 \sup_{\{(i_1,\ldots,i_d):
 f_n(i_1,\ldots,i_d)\ne0\}}
 \min_{1\leq r\leq d}
 \min\{p_{n,i_r},q_{n,i_r}\}
 \longrightarrow0.
 \label{bias_boundary}
\end{align}
Thus every nonzero interaction contains at least one coordinate whose
Bernoulli parameter approaches either zero or one.

In the first chaos, each nonzero interaction consists of a single
coordinate, and hence \eqref{bias_boundary} reduces to
\noi
\begin{align*}
 \max_{k:\Inf_k(F_n)>0}
 \min\{p_{n,k},q_{n,k}\}
 \longrightarrow0,
\end{align*}
which  is the bias rigidity
underlying the law-of-small-numbers regime.

\smallskip
(ii)
There is no analogous two-point bias parameter on Poisson space.
Instead, \eqref{P_zero_one} and \eqref{J_master} reveal a
unit-jump rigidity.  At first order, a Poisson atom is asymptotically
either ignored or counted with weight one; more generally, changes of
the functional having magnitude at least two are forced to become
negligible.  This does not require the control measure itself to become
small. See Example \ref{ex_first_P}.

Indeed, infinite divisibility gives, for every $m\geq1$,
$
 P_\lambda
 \stackrel{\mathrm{law}}{=}
 \sum_{j=1}^mP_{\lambda/m}^{(j)},
$
where the variables on the right-hand side are independent Poisson
random variables with mean $\lambda/m$.  Thus the Poisson law already
contains an intrinsic rare-event decomposition: by taking $m$ large,
it is represented as a sum of many independent counts having small
intensities.  In applications, spatial or combinatorial sparsity is a
natural way to realize the unit-jump mechanism, but, unlike the
Rademacher case, no additional abstract bias condition on the
underlying noise is required.
\end{remark}

\end{document}